\documentclass[11pt]{article}
\usepackage{enumerate}
\usepackage{mathbbold}
\usepackage{bbm}
\usepackage{amsfonts}
\usepackage[T1]{fontenc}
\usepackage{latexsym,amssymb,amsmath,amsfonts,amsthm}
\usepackage{graphicx}
\usepackage{subfigure}
\usepackage{epsf}
\usepackage{epstopdf}
\usepackage{amssymb}
\usepackage{mathrsfs}

\renewcommand{\theequation}{\arabic{section}.\arabic{equation}}

\newcommand{\R}{{\mathbb R}}

\newcommand{\C}{{\mathbb C}}
\newcommand{\Sp}{{\mathbb S}}
\newcommand{\ds}{\displaystyle}

\newcommand{\be}{\begin{eqnarray}}
\newcommand{\ben}{\begin{eqnarray*}}
\newcommand{\en}{\end{eqnarray}}
\newcommand{\enn}{\end{eqnarray*}}
\newcommand{\ba}{\backslash}
\newcommand{\pa}{\partial}

\newcommand{\ov}{\overline}

\newcommand{\eps}{\epsilon}

\newcommand{\om}{\omega}

\newcommand{\la}{\lambda}

\newcommand{\hth}{\hat{\theta}}
\newcommand{\hx}{\hat{x}}
\newtheorem{theorem}{Theorem}[section]
\newtheorem{lemma}[theorem]{Lemma}

\newtheorem{remark}[theorem]{Remark}

\begin{document}
\title{\bf Target reconstruction with a reference point scatterer using phaseless far field patterns}
\author{Xia Ji\thanks{LSEC, NCMIS and Academy of Mathematics and Systems Science, Chinese Academy of Sciences,
Beijing 100190, China. Email: jixia@lsec.cc.ac.cn (XJ)},
\and
Xiaodong Liu\thanks{NCMIS and Academy of Mathematics and Systems Science,
Chinese Academy of Sciences, Beijing 100190, China. Email: xdliu@amt.ac.cn (XL)},
\and
Bo Zhang\thanks{LSEC, NCMIS and Academy of Mathematics and Systems Science, Chinese Academy of Sciences,
Beijing 100190, China and School of Mathematical Sciences, University of Chinese Academy of Sciences,
Beijing 100049, China. Email: b.zhang@amt.ac.cn (BZ)}}
\date{}
\maketitle

\begin{abstract}
An important property of the phaseless far field patterns with incident plane waves is the translation invariance.
Thus it is impossible to reconstruct the location of the underlying scatterers. By adding a reference point
scatterer into the model, we design a novel direct sampling method using the phaseless data directly.
The reference point technique not only overcomes the translation invariance, but also brings a practical
phase retrieval algorithm.
Based on this, we propose a hybrid method combining the novel phase retrieval algorithm and the classical
direct sampling methods. Numerical examples in two dimensions are presented to demonstrate their effectiveness
and robustness.

\vspace{.2in}

\vspace{.2in} {\bf AMS subject classifications:}
35P25, 45Q05, 78A46, 74B05

\end{abstract}

\section{Introduction}

The {\em inverse scattering theory} has been a fast-developing area for the past thirty years.
Applications of inverse scattering problems occur in many areas such as radar, nondestructive testing,
medical imaging, geophysical prospection and remote sensing. Due to their applications, the inverse
scattering problems have attracted more and more attention, and significant progress has been made
for both the mathematical theories and numerical approaches \cite{CC2014,CK,Isakov,Kirsch,KirschGrinberg}.

In many cases of practical interest, it is very difficult and expensive to obtain the phased data,
while the phaseless data is much easier to be achieved.
Unfortunately, the reconstructions with phaseless data are highly nonlinear and much more severely
ill-posed \cite{AmmariChowZou}.
By adding a reference point scatterer into the scattering system, we introduce a direct sampling method
using the corresponding phaseless far field data directly.
Using at most three different scattering strength, we propose a novel phase retrieval scheme.
We show that, if the point scatterer is far away from the unknown scatterers, such a phase retrieval scheme
is Lipschitz stable with respect to the measurement noise. Based on this, we propose certain fast and robust
algorithms for scatterer reconstructions by combining classical sampling methods.

We begin with the formulations of the acoustic scattering problems. Let $k=\om/c>0$ be the wave number of
a time harmonic wave where $\om>0$ and $c>0$ denote the frequency and sound
speed, respectively. Let $D\subset\R^n (n=2,\, 3)$ be an open and bounded domain with
Lipschitz-boundary $\pa D$ such that the exterior $\R^n\ba\ov{D}$ is connected.
Furthermore, let the incident field $u^i$ be a plane wave of the form
\be\label{incidenwave}
u^i(x)\ =\ u^i(x,\hth) = e^{ikx\cdot \hth},\quad x\in\R^n\,,
\en
where $\hth\in\Sp^{n-1}$ denotes the direction of the incident wave and $S^{n-1}:=\{x\in\R^n:|x|=1\}$ is
the unit sphere in $\R^n$.
Then the scattering problem for the inhomogeneous medium is to find the total field $u=u^i+u^s$ such that
\be
\label{HemEqumedium}\Delta u + k^2 (1+q)u = 0\quad \mbox{in }\R^n,\\
\label{Srcmedium}\lim_{r:=|x|\rightarrow\infty}r^{\frac{n-1}{2}}\left(\frac{\pa u^{s}}{\pa r}-iku^{s}\right) =\,0,
\en
where $q\in L^{\infty}(\R^n)$ such that the imaginary part $\Im (q)\geq 0$ and $q=0$ in $\R^n\ba\ov{D}$,
the Sommerfeld radiating condition \eqref{Srcmedium} holds uniformly with respect to all directions
$\hx:=x/|x|\in\Sp^{n-1}$.
If the scatterer $D$ is impenetrable, the direct scattering is to find the total field $u=u^i+u^s$ such that
\be
\label{HemEquobstacle}\Delta u + k^2 u = 0\quad \mbox{in }\R^n\ba\ov{D},\\
\label{Bc}\mathcal{B}(u) = 0\quad\mbox{on }\pa D,\\
\label{Srcobstacle}\lim_{r:=|x|\rightarrow\infty}r^{\frac{n-1}{2}}\left(\frac{\pa u^{s}}{\pa r}-iku^{s}\right) =\,0,
\en
where $\mathcal{B}$ denotes one of the following three boundary conditions
\ben
(1)\,\mathcal{B}(u):=u\quad\mbox{on}\, \pa D;\qquad
(2)\,\mathcal{B}(u):=\frac{\pa u}{\pa\nu}\quad\mbox{on}\ \pa D;\qquad
(3)\,\mathcal{B}(u):=\frac{\pa u}{\pa\nu}+\la u\quad\mbox{on}\ \pa D
\enn
corresponding to the case when the scatterer $D$ is sound-soft, sound-hard, and of impedance type , respectively.
Here, $\nu$ is the unit outward normal to $\pa D$ and $\la\in L^{\infty}(\pa\,D)$ is the (complex valued)
impedance function such that $\Im(\la)\geq0$ almost everywhere on $\pa D$.
The well-posedness of the direct scattering problems \eqref{HemEqumedium}--\eqref{Srcmedium} and \eqref{HemEquobstacle}--\eqref{Srcobstacle}
have been established and can be found in \cite{CC2014,CK,KirschGrinberg,LiuZhangSiam,LiuZhang2012,LiuZhangHu,Mclean}.

Every radiating solution of the Helmholtz equation has the following asymptotic
behavior at infinity \cite{KirschGrinberg, LiuIP17}
\be\label{0asyrep}
u^s(x,\hth)
=\frac{e^{i\frac{\pi}{4}}}{\sqrt{8k\pi}}\left(e^{-i\frac{\pi}{4}}\sqrt{\frac{k}{2\pi}}\right)^{n-2}
\frac{e^{ikr}}{r^{\frac{n-1}{2}}}\left\{u^{\infty}_{D}(\hat{x},\hth)
+\mathcal{O}\left(\frac{1}{r}\right)\right\}\quad\mbox{as }\,r:=|x|\rightarrow\infty,
\en
uniformly with respect to all directions $\hx:=x/|x|\in\Sp^{n-1}$.
The complex valued function $u^{\infty}_{D}=u^{\infty}_{D}(\hx,\hth)$ defined on $\Sp^{n-1}$
is known as the scattering amplitude or far field pattern with $\hat{x}\in\Sp^{n-1}$ denoting the observation direction.
A wealth of results have been obtained on determining $D$ from the knowledge of the far field pattern $u^{\infty}_{D}$.
We refer to the standard monographs \cite{CC2014,CK,Isakov,Kirsch,KirschGrinberg}.
In practice, it is not always the case that the information about the full far field pattern
is known, but instead only its modulus might be given. Thus we are interested in
the following inverse problem:

{\bf (IP1):\quad\em Determine $D$ from the knowledge of phaseless far field pattern $|u^{\infty}_{D}|$.}

A well known difficulty for {\bf (IP1)} is that it is impossible to recover the location of a scatterer
only from the phaseless far field pattern due to the translation invariance.
Specifically, for the shifted obstacle $D_{h}:=\{x+h:\;x\in D\}$, or the shifted refractive index $n_h(x):=n(x-h)$
with a fixed vector $h\in \R^{n}$, the corresponding far field pattern $u^{\infty}_{D_h}$ satisfies the equality \cite{KR97,KS,LiuSeo}:
\be\label{translation}
u^{\infty}_{D_h}(\hx,\hth)=e^{ikh\cdot(\hth-\hx)}u^{\infty}_{D}(\hx,\hth),\quad \forall\, \hx, \hth\in\Sp^{n-1},
\en
i.e., the modulus of the far field pattern is invariant under translations.
Therefore, only the shape rather than the location may be uniquely determined by the modulus of the far field pattern.
In many corresponding uniqueness results with full far field patterns, the proofs heavily rely on the fact that
the far field pattern $u^{\infty}_{D}$ uniquely determines the scattered wave $u^s$, i.e., Rellich's lemma.
If it is known a priori that the scatterer is a sound-soft ball centered at the origin, uniqueness is established
to determine the radius of the ball by a single phaseless far field datum in \cite{LZball}.
Rellich's lemma is avoided in this special case.
By investigating the high frequency asymptotics of the far-field pattern, it was proved in \cite{majda76}
that the shape of a general smooth convex sound-soft obstacle can be determined by
the modulus of the far-field pattern associated with one plane wave as the incident field.
Up to now, no uniqueness results are available in determining general scatterers with the modulus of
the far field pattern generated by one incident plane wave, $|u^{\infty}_{D}(\hx,\hth)|,\hx,\hth\in\Sp^{n-1}$,
even with the translation invariance taken into account.
Initial effort was focused on the shape reconstruction numerically.
Indeed, many efficient numerical implementations \cite{AmmariChowZou,DZG,KarageorghisJohanssonLesnic,KR97,I07,IK10,Lee}
imply that shape reconstruction from the phaseless far field pattern is possible.
However, these methods are mainly iterative schemes based on the integral equations, and thus rely heavily on a priori information about the scatterer and are computationally expensive.

In recent years, considerable effort has been made to avoid using phaseless far field data or to break the translation invariance. Most of the works focus on the case that the point sources are scattered and the phaseless total/scattered
fields are measured, where the translation invariance property does not
hold \cite{ChengHuang,Klibanov14,Klibanov17,KlibanovRomanov17,KlibanovRomanov17-SIAM}.
The other possible way to break the translation invariance property is to consider superpositions of several plane waves
rather than one plane wave as incident fields.
The first breakthrough is given in \cite{ZZ-jcp17}, where the authors proved that the translation invariance property
of the phaseless far field pattern can be broken if superpositions of two plane waves are used as the incident fields
for all wave numbers in a finite interval. Further, a recursive Newton-type iteration algorithm in frequencies is also developed to numerically recover both the location and the shape of the scatterer $D$ simultaneously from
multi-frequency phaseless far-field data.
In a recent work \cite{XZZ-SIAM}, it was proved, under certain conditions on the scatterer, that
the scatterer can be uniquely determined by the phaseless far-field patterns generated by infinitely many sets
of superpositions of two plane waves with different directions at a fixed frequency.
A fast imaging algorithm was also developed in \cite{ZZ18} to numerically recover the
scattering obstacles from the phaseless far-field data at a fixed frequency associated with infinitely many sets
of superpositions of two plane waves with different directions.
Recently, the a priori assumption on the scatterers introduced in \cite{XZZ-SIAM} was removed in \cite{XZZ18}
by adding a known reference ball to the scattering system in conjunction with a simple technique based on Rellich's
lemma and Green's representation formula for the scattering solutions.
In addition, by adding a reference ball to the scattering system uniqueness results were obtained
in \cite{ZG} for inverse scattering with phaseless far-field data corresponding to superpositions of
an plane wave and point sources as the incident fields. Accordingly, a nonlinear integral equation method
was also developed in \cite{DZG} to reconstruct both the location and the shape of a scattering obstacle
from such phaseless far-field data in two dimensions.

The reference ball technique dates back to \cite{LLZ} (see also \cite{QinLiu}), where such a technique is
used to avoid eigenvalues and choose a cut-off value for the linear sampling method. To enhance the
interaction between the reference ball and the unknown target, the ball chosen in \cite{LLZ} can not be too
small or too far away from the unknown target.
In this paper, we still consider scattering of a single plane wave given in \eqref{incidenwave}, but add a known
reference point scatterer into the scattering system, which is different from \cite{XZZ18,ZG,DZG}.
Also, different from \cite{LLZ}, we expect the interaction between the reference point scatterer and the unknown
target is as weak as possible and thus choose the point far away from the unknown target.
Actually, numerical examples show that the interaction is very weak, even though the point scatterer is
close to the target. The reference point technique not only breaks the translation invariance of the phaseless far-field
patters, but also gives a novel phase retrieval method. This makes it possible to determine the unknown target
by combining the classical scatterer reconstruction methods with the phased data. We also want to strength that
our methods proposed in the next sections are independent of any a priori geometrical or physical information
on the unknown target.

The remaining part of the work is organized as follows. In Section \ref{Pre}, we first recall the scattering
of plane waves by a point scatterer, and then consider the scattering of plane waves by a combination of
the underlying scatterer $D$ and a given point scatterer located at $z_0\in\R^n\ba\ov{D}$.
Based on this, we propose a new inverse problem to determine the scatterer $D$ with the corresponding
phaseless far field data. A simple, fast and stable phase retrieval technique is then proposed.
Section \ref{DSMs} is devoted to some direct sampling methods for scatterer reconstructions,
which make no explicit assumptions on boundary conditions or topological properties of
the scatterer $D$. These algorithms are then verified in Section \ref{NumExamples} by extensive examples
in two dimensions.

\section{Preliminaries}
\label{Pre}\setcounter{equation}{0}

\subsection{Scattering of plane waves by a point scatterer}

First, we recall the scattering of plane waves by a point scatterer \cite{Foldy}.
We consider a point scatterer located at $z_0\in\R^{n}$ in the homogeneous space $\R^{n}$.
An incident plane wave $u^{i}$ of the form \eqref{incidenwave} is scattered by the target at $z_0$.
Recall that the fundamental solution $\Phi(x,y), x,y\in \R^n, x\neq y,$ of the Helmholtz equation is given by
\be\label{Phi}
\Phi(x,y):=\left\{
         \begin{array}{ll}
         \ds\frac{ik}{4\pi}h^{(1)}_0(k|x-y|)=\frac{e^{ik|x-y|}}{4\pi|x-y|}, & n=3, \\
         \ds\frac{i}{4}H^{(1)}_0(k|x-y|), & n=2,
         \end{array}
         \right.
\en
where $h^{(1)}_0$ and $H^{(1)}_0$ are, respectively, spherical Hankel function and Hankel function of
the first kind and order zero.
Then the scattered field $u^{s}_{z_0}$ is given by
\be\label{usz0}
u^{s}_{z_0}(x,\hth,\tau)=\tau u^{i}(z_0,\hth)\Phi(x,z_0).
\en
Here, $\tau\in\C$ is the scattering strength of the target.
From the asymptotic behavior of $\Phi(x,y)$ we deduce that the corresponding far field pattern is given by
\be\label{uinfz0}
u^{\infty}_{z_0}(\hx,\hth,\tau)=\tau u^{i}(z_0,\hth)e^{-ikz_0\cdot\hx} = \tau e^{ikz_0\cdot(\hth-\hx)},
\quad \hx, \hth\in \Sp^{n-1},\,\tau\in\C.
\en
Furthermore, by the representation \eqref{uinfz0} it is easy to deduce that for $h\in\R^n$,
\be\label{TRz0}
u^{\infty}_{z_0+h}(\hx,\hth,\tau)=e^{ikh\cdot(\hth-\hx)}u^{\infty}_{z_0}(\hx,\hth,\tau),
\quad \hx, \hth\in \Sp^{n-1},\,\tau\in\C.
\en
Then the translation relation for the phaseless data $|u^{\infty}_{z_0}|$ also holds, i.e.,
given $h\in\R^n$ we have
\ben
\big|u^{\infty}_{z_0+h}(\hx,\hth,\tau)\big|=\big|u^{\infty}_{z_0}(\hx,\hth,\tau)\big|,
\quad \hx, \hth\in \Sp^{n-1},\,\tau\in\C.
\enn

\subsection{New scattering system with a given point scatterer}

Let $z_0\in\R^n\ba\ov{D}$ be a fixed point outside $D$. By adding a point scatterer into the underlying
scattering system, we consider the new scattering system by $D\cup \{z_0\}$.
In the sequel, for an incident plane wave $u^i(x) = u^i(x, \hth) = e^{ik x\cdot\hth}$ we will indicate
the dependence of the scattered field and its far field pattern on
the incident direction $\hth$ and the scattering strength $\tau$ by writing, respectively,
$u^s_{D\cup\{z_0\}}(x, \hth, \tau)$ and $u^{\infty}_{D\cup\{z_0\}}(\hx,\hth,\tau)$.
Since the point scatterer is given in advance, the inverse problem considered is modified as follows.\\

{\bf (IP2):\quad\em Determine $D$ from the knowledge of phaseless far field pattern
$\big|u^{\infty}_{D\cup\{z_0\}}\big|$.}\\

Note that if $\tau=0$ then {\bf (IP2)} is reduce to {\bf (IP1)}.
Following the arguments given in \cite{KR97,KS,LiuSeo}, it is easy to check that the translation
invariance property of the phaseless far field data $\big|u^{\infty}_{D\cup\{z_0\}}\big|$ also holds,
i.e., for any $h\in\R^n$, we have
\ben
\big|u^{\infty}_{D_{h}\cup\{z_0+h\}}(\hx,\hth,\tau)\big|=\big|u^{\infty}_{D\cup\{z_0\}}(\hx,\hth,\tau)\big|,
\quad\forall \hx, \hth\in \Sp^{n-1}, \tau\in\C.
\enn
However, the reference point $z_0$ is given in advance, and this makes it possible to determine $D$ from
the phaseless data $\big|u^{\infty}_{D\cup\{z_0\}}\big|$.

It is well known that the nature of the scatterer $D$ can be uniquely determined by the phased far field patterns $u^{\infty}_{D}$ \cite{CK}.
In our subsequent analysis, we try to retrieve these phased data from the phaseless far field patterns $\big|u^{\infty}_{D\cup\{z_0\}}\big|$.
To do so, we expect that the interaction between the unknown target $D$ and the given point scatterer is
as weak as possible.
Such a fact can be achieved by choosing the reference point scatterer far away from the target.
Actually, this is verified by the following Theorem \ref{weakinteraction}.

For any $\varphi\in H^{-1/2}(\pa D)$ and $\psi\in H^{1/2}(\pa D)$,
the single-layer potential is defined by
\ben
(\mathcal{S}\varphi)(x):=\int_{\pa D}\Phi(x,y)\varphi(y)ds(y),\quad x\in\R^n\ba{\pa D},
\enn
and the double-layer potential is defined by
\ben
(\mathcal{K}\psi)(x):=\int_{\pa D}\frac{\pa\Phi(x,y)}{\pa\nu(y)}\psi(y)ds(y),\quad x\in\R^n\ba{\pa D},
\enn
respectively. It is shown in \cite{Mclean} that the potentials
$\mathcal{S}: H^{-1/2}(\pa D)\rightarrow H^{1}_{loc}(\R^n\ba\pa D)$,
$\mathcal{K}: H^{1/2}(\pa D)\rightarrow H^{1}_{loc}(\R^n\ba\ov{D})$ are well defined.
We also define the restriction of $\mathcal{S}$ and $\mathcal{K}$ to the  boundary $\pa D$ by
 \be
\label{S} (S\varphi)(x):= \int_{\pa D}\Phi(x,y)\varphi(y)ds(y),\quad  x\in \pa D,\\
\label{K} (K\psi)(x):= \int_{\pa D}\frac{\pa\Phi(x,y)}{\pa\nu(y)}\psi(y)ds(y),\quad  x\in \pa D.
 \en
We refer to \cite{Mclean} for the properties of the boundary operators
$S: H^{-1/2}(\pa D)\rightarrow H^{1/2}(\pa D)$ and $K: H^{1/2}(\pa D)\rightarrow H^{1/2}(\pa D)$.

\begin{theorem}\label{weakinteraction}
Let $z_0$ be a point outside $D$ such that the distance $\rho:=dist (z_0, D)$ is large enough. Then we have
\be\label{uinfweaksca}
u^{\infty}_{D\cup\{z_0\}}(\hx,\hth,\tau) = u^{\infty}_{D}(\hx,\hth) + u^{\infty}_{z_0}(\hx,\hth,\tau) + O\left(\rho^{\frac{1-n}{2}}\right),\quad\forall \hx, \hth\in \Sp^{n-1}, \tau\in\C.
\en
\end{theorem}

\begin{proof}
The integral equation method is used in our subsequent analysis. For simplicity, we only prove
the case when the scatterer $D$ is sound-soft. The other cases can be dealt with similarly.

We seek a solution in the form
\be\label{usDz0}
u^{s}_{D\cup\{z_0\}}(x,\hth,\tau)= (\mathcal{K}-i\eta\mathcal{S})\psi_{D\cup\{z_0\}} + u^{s}_{z_0}(x,\hth,\tau),\quad x\in\R^n\ba\ov{D\cup\{z_0\}},
\en
with a density $\psi_{D\cup\{z_0\}}\in H^{1/2}(\pa D)$ and a coupling parameter $\eta > 0$.
Then from the jump relation of the double layer potential we see that the representation $u^{s}_{D\cup\{z_0\}}$
given in \eqref{usDz0} solves the exterior Dirichlet boundary problem provided the density is a solution of
the integral equation
\ben
\Big({I}/{2}+K-i\eta S\Big)\psi_{D\cup\{z_0\}} = -(u^i+u^{s}_{z_0})\quad \mbox{on}\,\,\pa D.
\enn
Note that $I/2+K-i\eta S$ is bijective and the inverse $(I/2+K-i\eta S)^{-1}: H^{1/2}(\pa D)\rightarrow H^{1/2}(\pa D)$ is bounded \cite{CK, Mclean}. Therefore
\ben
\psi_{D\cup\{z_0\}} = -(I/2+K-i\eta S)^{-1}(u^i+u^{s}_{z_0}):=\psi_D - (I/2+K-i\eta S)^{-1}u^{s}_{z_0}.
\enn
Note that $u^s_D:=(\mathcal{K}-i\eta\mathcal{S})\psi_D$ is the radiating solution to the original scattering system
with a sound-soft obstacle $D$.
A straightforward calculation shows that the fundamental solution $\Phi$ satisfies Sommerfeld's finiteness condition
\ben
\Phi(x,y)=O\left(|x-y|^{\frac{1-n}{2}}\right), \quad |x-y|\rightarrow\infty.
\enn
Inserting this into the representation \eqref{usz0} of $u^{s}_{z_0}$ we see that
\ben
u^{s}_{z_0}\Big|_{\pa D}= O\left(\rho^{\frac{1-n}{2}}\right), \quad \rho\rightarrow\infty.
\enn
This implies that
\ben
\psi_{D\cup\{z_0\}} = \psi_D + O\left(\rho^{\frac{1-n}{2}}\right), \quad \rho\rightarrow\infty.
\enn
Inserting this into \eqref{usDz0}, we find that
\ben
u^s_{D\cup\{z_0\}}(x,\hth,\tau) = u^s_{D}(x,\hth)+ u^s_{z_0}(x,\hth,\tau) + O\left(\rho^{\frac{1-n}{2}}\right), \quad x\in\R^n\ba\ov{D\cup\{z_0\}}, \quad \rho\rightarrow\infty.
\enn
Then \eqref{uinfweaksca} follows by letting $|x|\rightarrow \infty$.
\end{proof}

\subsection{Phase retrieval}\label{subsec-phaseretrieval}

The following lemma may have its own interest.

\begin{lemma}\label{phaseretrieval}
Let $z_j:=x_j+iy_j,\,j=1,2,3,$ be three different complex numbers such that they are not collinear.
Then the complex number $z\in\C$ is uniquely determined by the distances $r_j=|z-z_j|,\,j=1,2,3$.
\end{lemma}

\begin{proof}
Denote by $Z_j=(x_j,y_j)$ the point in the plane corresponding to the given complex number $z_j$, $j=1,2,3.$
Define $Z=(x,y)$ to be the point corresponding to the unknown complex number $z$.
Then $Z$ locates on the spheres $\pa B_{r_j}(Z_j)$ centered at $Z_j$ with radius $r_j,\,j=1,2,3$.
Note that there are at most two points $Z_A$ and $Z_B$ located simultaneously on the two spheres
$\pa B_{r_j}(Z_j), j=1,2$, i.e.
\be\label{ZZ12}
|Z_A-Z_j|=|Z_B-Z_j|=r_j, \,j=1,2.
\en
If $Z_A=Z_B$, then we just take $Z=Z_A$; otherwise, we claim that only one of the two points $Z_A$
and $Z_B$ is the point $Z$ pursued. On the contrary, we have
\ben
|Z_A-Z_3|=|Z_B-Z_3|=r_3.
\enn
This, together with \eqref{ZZ12}, implies that the three points $Z_j,\,j=1,2,3,$ are located on
the perpendicular bisector of the line segment $Z_{A}Z_B$. This contradicts to
the assumption that $z_j,\,j=1,2,3,$ are not collinear. The proof is complete.
\end{proof}

Actually, Lemma \ref{phaseretrieval}
provides a novel phase retrieval technique which can be implemented easily. Using the same notations in Lemma \ref{phaseretrieval}, we have the following phase retrieval scheme.

\begin{figure}[htbp]
\centering
\includegraphics[width=3in]{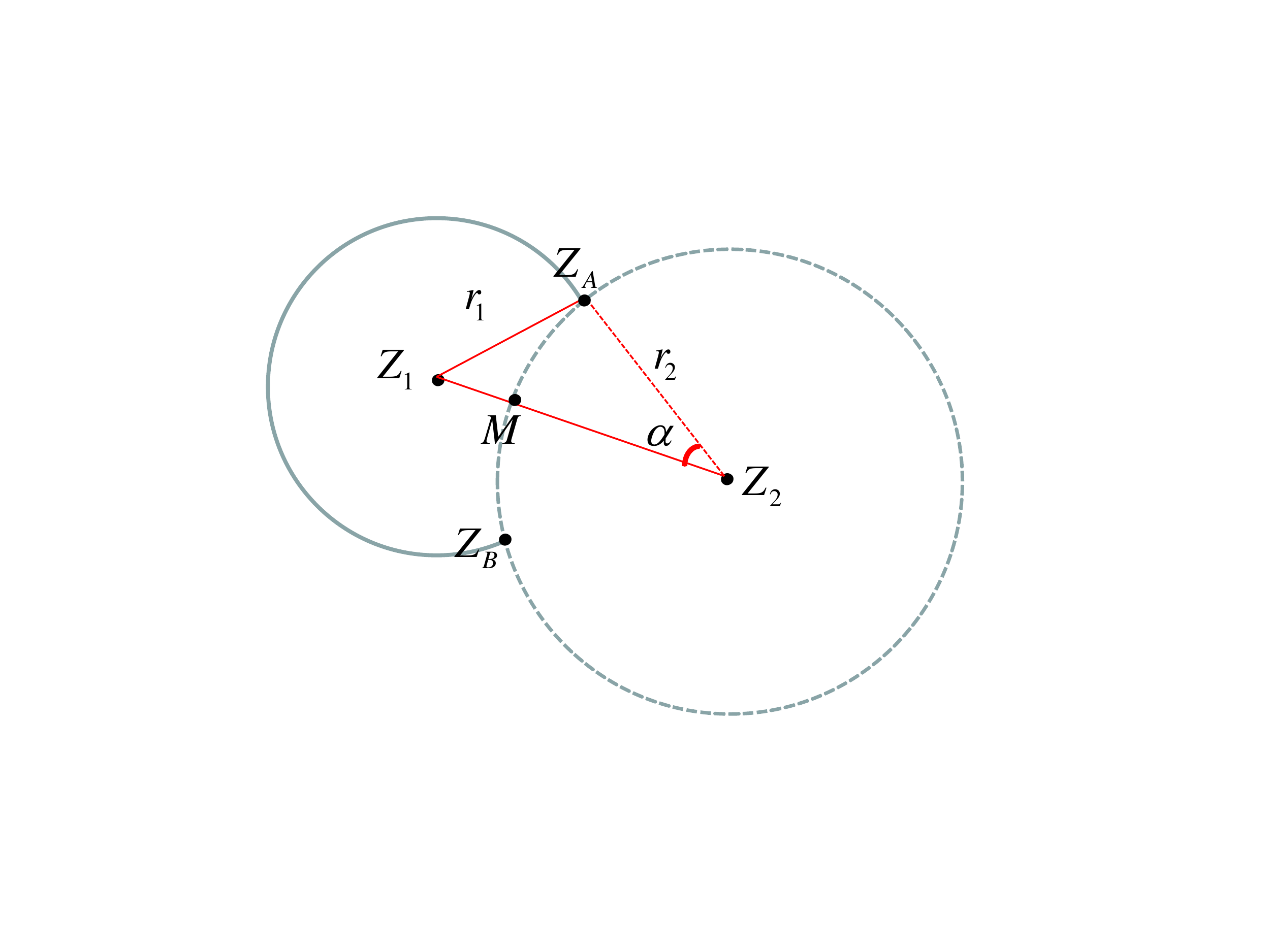}
\caption{Sketch map for phase retrieval scheme.}
\label{twodisks}
\end{figure}

{\bf Phase Retrieval Scheme. (Numerical simulation for Lemma \ref{phaseretrieval}.)}

\begin{enumerate}[(1)]{\em
\item {\bf\rm Collect the distances $r_j:=|z-z_j|$ with given complex numbers $z_j,\,j=1,2,3$. }
  If $r_j=0$ for some $j\in\{1,2,3\}$, then $Z=Z_j$; otherwise, go to the next step.
\item {\bf\rm Look for the point $M=(x_M, y_M)$.} As shown in Figure \ref{twodisks}, $M$ is the intersection
  of circle centered at $Z_2$ with radius $r_2$ and the ray $Z_2Z_1$ with initial point $Z_2$.
  Denote by $d_{1,2}:=|z_1-z_2|$ the distance between $Z_1$ and $Z_2$. Then
       \be\label{xMyM}
       x_M=\frac{r_2}{d_{1,2}}x_1+\frac{d_{1,2}-r_2}{d_{1,2}}x_2,\quad y_M=\frac{r_2}{d_{1,2}}y_1+\frac{d_{1,2}-r_2}{d_{1,2}}y_2.
       \en
  \item {\bf\rm Look for the points $Z_A=(x_A, y_A)$ and $Z_B=(x_B, y_B)$.} Note that $Z_A$ and $Z_B$
  are just two rotations of $M$ around the point $Z_2$. Let $\alpha\in [0,\pi]$ be the angle between
  the rays $Z_2Z_1$ and $Z_2Z_A$. Then, by the law of cosine we have
       \be\label{cosalpha}
       \cos \alpha=\frac{r_1^2-r_2^2-d_{1,2}^2}{2r_2d_{1,2}}.
       \en
  Noting that $\alpha\in [0,\pi]$ and $\sin^2 \alpha+\cos^2 \alpha=1$, we deduce that
  $\sin \alpha=\sqrt{1-\cos^2\alpha}$. Then
       \be
       \label{xA}x_A &=& x_2+\Re\{[(x_M-x_2)+i(y_M-y_2)]e^{-i\alpha}\},\\
       \label{yA}y_A &=& y_2+\Im\{[(x_M-x_2)+i(y_M-y_2)]e^{-i\alpha}\},\\
       \label{xB}x_B &=& x_2+\Re\{[(x_M-x_2)+i(y_M-y_2)]e^{i\alpha}\},\\
       \label{yB}y_B &=& y_2+\Im\{[(x_M-x_2)+i(y_M-y_2)]e^{i\alpha}\}.
       \en
  \item {\bf\rm Determine the point $Z$.} $Z=Z_A$ if the distance $|Z_AZ_3|=r_3$; otherwise, $Z=Z_B$.
  }
\end{enumerate}

\begin{remark}\label{ZStableRec} {\rm
Actually, the above scheme provides a stable phase retrieval algorithm. Indeed, let $\eps>0$ and assume that
\ben
|r_j^{\eps}-r_j|\leq\eps, \quad j=1,2,3.
\enn
Here, and throughout the paper, we use the subscript $\eps$ to denote the polluted data.
From \eqref{xMyM}, we deduce that
\be\label{xMyMestimate}
|x_M^{\eps}-x_M|=\frac{|x_1-x_2|}{d_{1,2}}|r_2^{\eps}-r_2|\leq \eps\quad\mbox{and}\quad |y_M^{\eps}-y_M|=\frac{|y_1-y_2|}{d_{1,2}}|r_2^{\eps}-r_2|\leq \eps.
\en
Similarly, \eqref{cosalpha} implies the existence of a constant $c_1>0$ depending on $Z_j, j=1,2,3$, such that
\ben
|e^{i\alpha^{\eps}}-e^{i\alpha}|\leq c_1\eps.
\enn
Combing this with \eqref{xMyMestimate} and \eqref{xA}-\eqref{yB}, we find that
there exists a constant $c_2>0$ depending on $Z_j, j=1,2,3$, such that
\ben
|x_{ii}^{\eps}-x_{ii}|\leq c_2\eps\quad\mbox{and}\quad |y_{ii}^{\eps}-y_{ii}|\leq c_2\eps,\quad ii=A,B.
\enn
Therefore, we have
\ben
|Z^{\eps}-Z|\leq \sqrt{2}c_2\eps.
\enn
This implies that our phase retrieval scheme is Lipschitz stable with respect to the measurement noise level $\eps$.
}
\end{remark}

Using the phase retrieval scheme, we wish to approximately reconstruct $u_{D}^{\infty}$ from the knowledge of the perturbed phaseless data
$\left|u^{\infty,\eps}_{D\cup\{z_0\}}\right|$ with a known error level
\be\label{uinfeps}
\left|\big|u^{\infty,\eps}_{D\cup\{z_0\}}\big|-\big|u^{\infty}_{D\cup\{z_0\}}\big|\right|\leq \eps.
\en

\begin{theorem}\label{phaseretrieval-stability}
Let $\tau_j\in\C, j=1,2,3$ be three scattering strengths with different principle arguments and $\rho$ be the distance between the point $z_0$ and the unknown target $D$. Under the measurement error estimate \eqref{uinfeps}, we have
\be\label{uDinfestimate}
\left|u^{\infty,\eps}_{D}-u^{\infty}_{D}\right|\leq c_3\eps+O\left(\rho^{\frac{1-n}{2}}\right),
\en
for some constant $c_3>0$ depending only on $\tau_j, j=1,2,3$.
\end{theorem}

\begin{proof}
Let
\ben
Z_j:=u^{\infty}_{z_0}(\hx,\hth,\tau_j)=\tau_je^{ikz_0\cdot(\hth-\hx)}, \quad j=1,2,3.
\enn
Then the assumption on the strengths implies that the three points $Z_j, j=1,2,3$ are not collinear. Define $r_j^{\eps}:=\left|u^{\infty,\eps}_{D\cup\{z_0\}}\right|$.
Using Theorem \ref{weakinteraction}, we have
\ben
r_j^{\eps}=\left|u^{\infty,\eps}_{D} + u^{\infty}_{z_0} + O\left(\rho^{\frac{1-n}{2}}\right)\right|,
\enn
where $\rho$ is the distance between the point $z_0$ and the unknown target $D$.
Then, following the arguments in Remark \ref{ZStableRec}, we have
\ben
\left|u^{\infty,\eps}_{D}-u^{\infty}_{D}+O\left(\rho^{\frac{1-n}{2}}\right)\right|\leq c_3\eps,
\enn
for some constant $c_3>0$ depending only on $\tau_j, j=1,2,3$. The stability estimate \eqref{uDinfestimate}
now follows by using the triangle inequality.
\end{proof}

Finally, we want to remark that Theorem \ref{phaseretrieval-stability} implies that our phase retrieval
scheme provides a stable method for the phase reconstruction.
This will also be verified by the numerical example in Section \ref{NumExamples}.

\subsection{Stability estimates for the inverse problems}

Stability of recovery of the scatterer is crucial for numerical algorithms. It has been proved in several
papers that the inverse problems with phased far field patterns are ill-posed. Stability was first
considered by Isakov \cite{Isakov92, Isakov93} for the determination of a sound-soft obstacle.
We refer to Potthast \cite{Potthast2000}, Cristo and Rondi \cite{CristoRondi} for the extension to
both the sound-soft and sound-hard obstacles. H$\ddot{a}$hner and Hohage \cite{HahnerHohage}
considered the stability estimate for the inhomogeneous medium case.

In this subsection, let $D_1, D_2$ be two scatterers, $\rho$ be the distance between the point $z_0$
and the unknown obstacles $D_1\cup D_2$ and $\tau_j\in\C, j=1,2,3,$ be three scattering strengths
with different principal arguments.
Let $|u^{\infty}_{D_1\cup\{z_0\}}|, |u^{\infty}_{D_2\cup\{z_0\}}|$ be the corresponding phaseless
far field patterns. Following the same arguments in Theorem \ref{phaseretrieval-stability},
we have the following stability estimate.

\begin{theorem}\label{phasestability}
If
\be\label{phaselessinf12}
\left||u^{\infty}_{D_1\cup\{z_0\}}(\hx,\hth,\tau)|-|u^{\infty}_{D_2\cup\{z_0\}}(\hx,\hth,\tau)|\right|<\eps,
\quad \forall \hx, \hth\in \Sp^{n-1}, \tau\in\{\tau_1, \tau_2,\tau_3\},
\en
then for sufficiently large $\rho$ we have
\be\label{phaseinf12}
\left|u^{\infty}_{D_1}(\hx,\hth)-u^{\infty}_{D_2}(\hx,\hth)\right|<c_3\eps + O\left(\rho^{\frac{1-n}{2}}\right),
\quad \forall \hx, \hth\in \Sp^{n-1}.
\en
where $c_3>0$ is a constant depending only on $\tau_j, j=1,2,3$.
\end{theorem}

Combining Theorem \ref{phasestability} and Theorem 1 in \cite{Isakov92}, Theorem 15 in \cite{Potthast2000}
and Theorem 1.2 in \cite{HahnerHohage}, respectively, we immediately obtain the following stability
estimate for scatterer reconstruction with phaseless far field patterns.

\begin{theorem}
Denote by $\delta=\delta(\eps,\rho)$ the right-hand side of \eqref{phaseinf12}. The following scatterer
reconstruction stabilities hold.
\begin{enumerate}[(1)]
\item Assume that $D_m=\{x\in\R^3: |x|<r_m(\theta)\}$ is star-shaped with
\ben
\|r_m\|_{C^{2,\alpha}(S^2)}<1/R_1, \quad 1/R_1< r_m < R_0, \quad m=1,2.
\enn
For $k<\pi/R_0$, if \eqref{phaselessinf12} holds for any fixed $\hth\in S^2$, then the Hausdorff
distance between $D_1$ and $D_2$ satisfies that
\ben
dist(D_1, D_2)< C\left(ln\Big(-ln(\delta)\Big)\right)^{-1/C},
\enn
where $C$ is a constant depending only on $R_1$.
\item Assume that $D_m\subset B_R(0), m=1,2,$ are sound-soft or sound-hard obstacle with $C^2$ boundary
  satisfying the exterior cone condition with angle $\beta$. If \eqref{phaselessinf12} holds for
  all $\hx,\hth\in S^{n-1}$, then the Hausdorff distance between the convex hulls
  $\mathscr{H}(D_1)$ and $\mathscr{H}(D_2)$ satisfies the estimate
  \ben
  dist(\mathscr{H}(D_1), \mathscr{H}(D_2))\leq \frac{C}{\left|ln(\delta)\right|^\alpha},
  \enn
  where the constants $C>0$ and $0<\alpha<1$ uniformly depend only on $R$ and $\beta$.
  \item Assume that $q_m\in H^s(\R^3)$ for some fixed $s>3/2$, $supp(q_m)\subset B_1$ and
  $\|q_m\|_{H^s}< C_q$ for some fixed constant $C_q>0$, $m=1,2$. For any fixed constant
  $\eps_0\in\left(0,({2s-3})/({2s+3})\right)$, the maximum norm of $q_1-q_2$ can be estimated as
  \ben
  \|q_1-q_2\|_{\infty}\leq C\left(-\widetilde{ln}(16\pi^2\delta)\right)^{\eps_0-({2s-3})/({2s+3})},
  \enn
where $C$ depends only on $C_q, \eps_0$, and $\widetilde{ln}(t):= ln(t)$ for $t<1/e$ and
$\widetilde{ln}(t):=-1$ otherwise.
\end{enumerate}
\end{theorem}

\section{Direct sampling methods}\label{DSMs}
\setcounter{equation}{0}

In this section, we investigate the numerical method for reconstruction of $D$ by using phaseless
far field data $\big|u^{\infty}_{D\cup\{z_0\}}\big|$. We will focus on designing a direct sampling method
which do not need any a priori information on the geometry and physical properties of the obstacle.
Roughly speaking, a direct sampling method chooses an appropriate indicator function $I(z),\;z\in\R^n$,
such that its value has an obvious change across the boundary of the scatterers.

We first introduce two auxiliary functions
\be\label{GA}
G(z,\hth):=\int_{S^{n-1}}u^{\infty}_{D}(\hx,\hth)e^{ik\hx\cdot z}ds(\hx)\;\;\;
A(z):=\int_{S^{n-1}}G(z,\hth)e^{-ik\hth\cdot\,z}ds(\hth),\quad\,z\in\R^n.
\en
By the well-known Riemann-Lebesgue Lemma, both $G$ and $A$ tend to $0$ as $|z|\rightarrow\infty$.
For the scattering problems \eqref{HemEqumedium}-\eqref{Srcmedium} and  \eqref{HemEquobstacle}-\eqref{Srcobstacle},
it is well known that the far field pattern $u^{\infty}_{D}$ has the following form (cf. \cite{KirschGrinberg})
\ben
u^{\infty}_{D}(\hx,\hth)=\int_{\pa D}\left\{u^s(y,\hth)\frac{\pa e^{-ik\hx\cdot y}}{\pa\nu(y)}
-\frac{\pa u^s(y,\hth)}{\pa\nu}e^{-ik\hx\cdot y}\right\}ds(y),\quad (\hx,\hth)\in\Sp^{n-1}.
\enn
Inserting this into \eqref{GA}, integrating by parts and using the well-known Funk-Hecke formula \cite{CK,LiuIP17},
we deduce that
\be\label{G}
&&G(z,\hth)\cr
&&=\int_{\Sp^{n-1}}\int_{\pa D}\left\{u^s(y,\hth)\frac{\pa e^{-ik\hx\cdot (y-z)}}{\pa\nu(y)}
-\frac{\pa u^s(y,\hth)}{\pa\nu}e^{-ik\hx\cdot(y-z)}\right\}ds(y)ds(\hx)\cr
&&=\int_{\pa D}\left\{-iku^s(y,\hth)\nu(y)\cdot\int_{\Sp^{n-1}}\hx e^{-ik\hx\cdot (y-z)}ds(\hx)
-\frac{\pa u^s(y,\hth)}{\pa\nu}\int_{\Sp^{n-1}}e^{-ik\hx\cdot(y-z)}ds(\hx)\right\}ds(y)\cr
&&=\int_{\pa D}\left\{-ik\mu_1u^s(y,\hth)\nu(y)\cdot\frac{y-z}{|y-z|}f_1(k|y-z|)
-\mu_0\frac{\pa u^s(y,\hth)}{\pa\nu}f_0(k|y-z|)\right\}ds(y),
\en
where
\ben
\mu_{\alpha}=\left\{
               \begin{array}{ll}
                 {2\pi}{i^{-\alpha}}, & \hbox{$n=2$,} \\
                 {4\pi}{i^{-\alpha}}, & \hbox{$n=3$}
               \end{array}
             \right.
\qquad\mbox{and}\qquad
f_{\alpha}(t)=\left\{
               \begin{array}{ll}
                 J_{\alpha}(t), & \hbox{$n=2$,} \\
                 j_{\alpha}(t), & \hbox{$n=3$}
               \end{array}
             \right.
\enn
with $J_{\alpha}$ and $j_{\alpha}$ being the Bessel functions and spherical Bessel functions of order $\alpha$,
respectively. This implies that $G$ is a superposition of the Bessel functions $f_0$ and $f_1$.
We thus expect that $G$ (and therefore $A$) decays like Bessel functions as the sampling points away from
the boundary of the scatterer.

Then one may look for the scatterers by using the following indicators \cite{LiZou, LiuIP17, Potthast2010}
with phased far field patterns,
\be\label{indicator23}
{\bf I_2}(z)=|A(z)| \quad\mbox{and}\quad {\bf I_3}(z,\hth)=|G(z,\hth)|,
\en
where $A$ and $G$ are given in \eqref{GA}. In \cite{LiuIP17}, it has been showed that the indicator ${\bf I_2}$
has a positive lower bound for sampling points
inside the scatterer, and decays like Bessel functions as the sampling points away from the boundary.
If the size of the scatterer $D$ is small enough (compared with the wavelength),
${\bf I_3}$ takes its local maximum at the location of the scatterer \cite{LiZou,Potthast2010}.

Consider now the case of phaseless far field measurements.
Using \eqref{uinfz0} and \eqref{uinfweaksca}, we have
\be\label{F}
&&\mathcal {F}(\hx,\hth,z_0,\tau)\cr
&:=&|u^{\infty}_{D\cup\{z_0\}}(\hx, \hth, \tau)|^2-|u^{\infty}_{D}(\hx, \hth)|^2-|\tau|^2\cr
&=&\Big|u^{\infty}_{D}(\hx, \hth)+u^{\infty}_{z_0}(\hx,\hth,\tau)
+ O\Big(\rho^{\frac{1-n}{2}}\Big)\Big|^2-|u^{\infty}_{D}(\hx, \hth)|^2-|\tau|^2\cr
&=&u^{\infty}_{D}(\hx, \hth)\ov{\tau}e^{-ikz_0\cdot(\hth-\hx)}
+\ov{u^{\infty}_{D}(\hx,\hth)}\tau e^{ikz_0\cdot(\hth-\hx)}
+ O\Big(\rho^{\frac{1-n}{2}}\Big), \quad (\hx,\hth)\in\Sp^{n-1},\,\tau\in\C.
\en
Denote by $\Theta$ a finite set with finitely many incident directions as elements.
Then, for any fixed $\tau\in \C\ba\{0\}$ and $z_0\in\R^n\ba\ov{D}$, we introduce the following
two indicators
\be
\label{Indicator01}{\bf I^{\Theta}_{z_0}}(z)&:=&\left|\sum_{\hth\in\Theta}\int_{\Sp^{n-1}} \mathcal {F}(\hx,\hth,z_0,\tau)\cos[k\hx\cdot(z-z_0)]ds(\hx)\right|,\quad\,z\in\R^n, \\
\label{Indicator0f}{\bf I_{z_0}}(z)&:=&\Big|\int_{\Sp^{n-1}}\int_{\Sp^{n-1}}
\mathcal {F}(\hx,\hth,z_0,\tau)\cos[k(\hx-\hth)\cdot(z-z_0)]ds(\hx)ds(\hth)\Big|,\quad\,z\in\R^n.\quad
\en
Insert \eqref{F} into \eqref{Indicator01}-\eqref{Indicator0f}. Then a straightforward calculation shows that
\ben
{\bf I^{\Theta}_{z_0}}(z)=\left|\sum_{\hth\in\Theta}( V_{z_0}(z,\hth)+\ov{V_{z_0}(z,\hth)}) \right|+ O\Big(\rho^{\frac{1-n}{2}}\Big)
\quad\mbox{and}\quad {\bf I_{z_0}}(z)=\big| W_{z_0}(z)+\ov{W_{z_0}(z)} \big|+ O\Big(\rho^{\frac{1-n}{2}}\Big)
\enn
with
\ben
V_{z_0}(z,\hth):=\frac{\ov{\tau}e^{-ik\hth\cdot z_0}}{2}\Big[G(z,\hth)+G(2z_0-z,\hth)\Big]
\,\,\mbox{and}\,\,
W_{z_0}(z):=\frac{\ov{\tau}}{2}\Big[A(z)+A(2z_0-z)\Big],\, z\in\R^n.
\enn
Let $D(z_0)$ be the point symmetric domain of $D$ with respect to $z_0$.
If the size of the scatterer $D$ is small enough, from \eqref{G} we expect that the indicator
${\bf I^{\Theta}_{z_0}}$ takes its local maximum on the locations of $D$ and $D(z_0)$.
For extended scatterer $D$, from the behavior of the indicator $A$ we expect that
the indicator ${\bf I_{z_0}}$ takes its maximum on or near the boundary $\pa D\cup\pa D(z_0)$.

Note that the indicator ${\bf I^{\Theta}_{z_0}}$/${\bf I_{z_0}}$ produces a false scatterer $D(z_0)$.
However, since we have the freedom to choose the point $z_0$, we can always choose $z_0$ such that
the false domain $D(z_0)$ located outside our interested searching domain.
One may also overcome this problem by considering another indicator
${\bf I^{\Theta}_{z_1}}$/${\bf I_{z_1}}$ with $z_1\in \R^n\ba\ov{D}$ and $z_1\neq z_0$.

{\bf Scatterer Reconstruction Scheme One.}
\begin{enumerate}[(1)]{\em
  \item Collect the phaseless data set $\big\{|u^{\infty}_{D\cup\{z_0\}}(\hx,\hth,\tau)|:\,
  (\hx,\hth)\in\Sp^{n-1}, \tau\in\{0,\tau_1\}\big\}$.
  \item Select a sampling region in $\R^{n}$ with a fine mesh $\mathcal {T}$ containing the scatterer $D$,
  \item Compute the indicator functional ${\bf I_{z_0}}(z)$ (or ${\bf I^{\Theta}_{z_0}}$ in
  the case of small scatterers) for all sampling point $z\in\mathcal{T}$,
  \item Plot the indicator functional ${\bf I_{z_0}}(z)$ (or ${\bf I^{\Theta}_{z_0}}$
  in the case of small scatterers).}
\end{enumerate}

Using the Phase Retrieval Scheme proposed in the previous section, we obtain the approximate phased far field pattern $u_{D}^{\infty}$. Then we have the second scatterer reconstruction algorithm.

{\bf Scatterer Reconstruction Scheme Two.}
\begin{enumerate}[(1)]{\em
  \item Collect the phaseless data set
  $\big\{|u^{\infty}_{D\cup\{z_0\}}(\hx,\hth,\tau)|:\,(\hx,\hth)\in\Sp^{n-1},\tau\in\{\tau_1,\tau_2,\tau_3\}\big\},$
  \item Use the {\bf Phase Retrieval Scheme} to obtain the phased far field patterns $u^{\infty}_{D}(\hx, \hth)$
  for all $(\hx,\hth)\in\Sp^{n-1}$,
  \item Select a sampling region in $\R^{n}$ with a fine mesh $\mathcal{T}$ containing $D$,
  \item Compute the indicator functional ${\bf I_2}(z)$ (or ${\bf I_3}(z,\hth)$ in the case of small
  scatterers with a fixed incident direction $\hth$) for all sampling point $z\in\mathcal {T}$,
  \item Plot the indicator functional ${\bf I_2}(z)$ (or ${\bf I_3}(z,\hth)$ in the case of small scatterers).}
\end{enumerate}

\section{Numerical examples and discussions}
\label{NumExamples}
\setcounter{equation}{0}

Now we present a variety of numerical examples in two dimensions to illustrate the applicability
and effectiveness of our sampling methods.
There are totally nine groups of numerical tests to be considered, and they are
respectively referred to as {\bf ${\bf I_{z_0}}$-Soft, ${\bf I_{z_0}}$-Multiple,
${\bf I_{z_0}}$-MultiScalar, ${\bf I^{\Theta}_{z_0}}$-Small, PhaseRetrieval,
${\bf I_{2}}$-Soft, ${\bf I_{2}}$-Multiple, ${\bf I_{2}}$-MultiScalar,}
and {\bf ${\bf I_{3}}$-Small}.
The boundaries of the scatterers used in our numerical experiments are parameterized as follows
\be
\label{kite}&\mbox{\rm Kite:}&\quad x(t)\ =(a,b)+\ (\cos t+0.65\cos 2t-0.65, 1.5\sin t),\quad 0\leq t\leq2\pi,\\
\label{peanut}&\mbox{\rm Peanut:}&\quad x(t)\ =(a,b)+\, 2\sqrt{3\cos^2 t+1}(\cos t, \sin t),\quad 0\leq t\leq2\pi,\\
\label{pear}&\mbox{\rm Pear:}&\quad x(t)\ =(a,b)+\, (2+0.3\cos 3t)(\cos t, \sin t),\quad 0\leq t\leq2\pi,\\
\label{circle}&\mbox{\rm Circle:}&\quad x(t)\ =(a,b)+r\, (\cos t, \sin t),\quad 0\leq t\leq2\pi,\quad
\en
with $(a,b)$ be the location of the scatterer which may be different in different examples
and $r$ be the radius of the circle.

\begin{figure}[htbp]
  \centering
  \subfigure[\textbf{Kite}]{
    \includegraphics[width=1.4in]{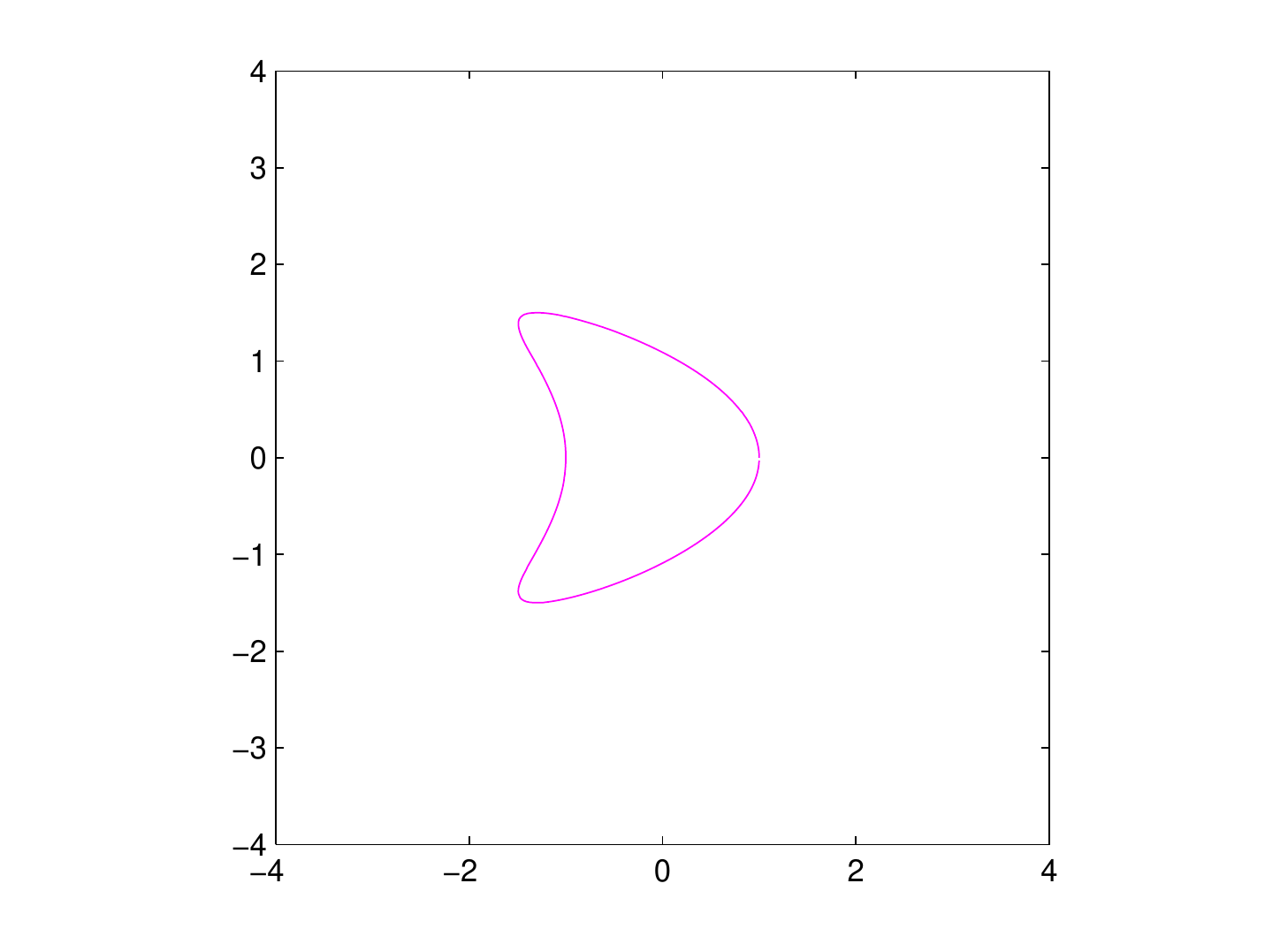}}
  \subfigure[\textbf{Peanut}]{
    \includegraphics[width=1.4in]{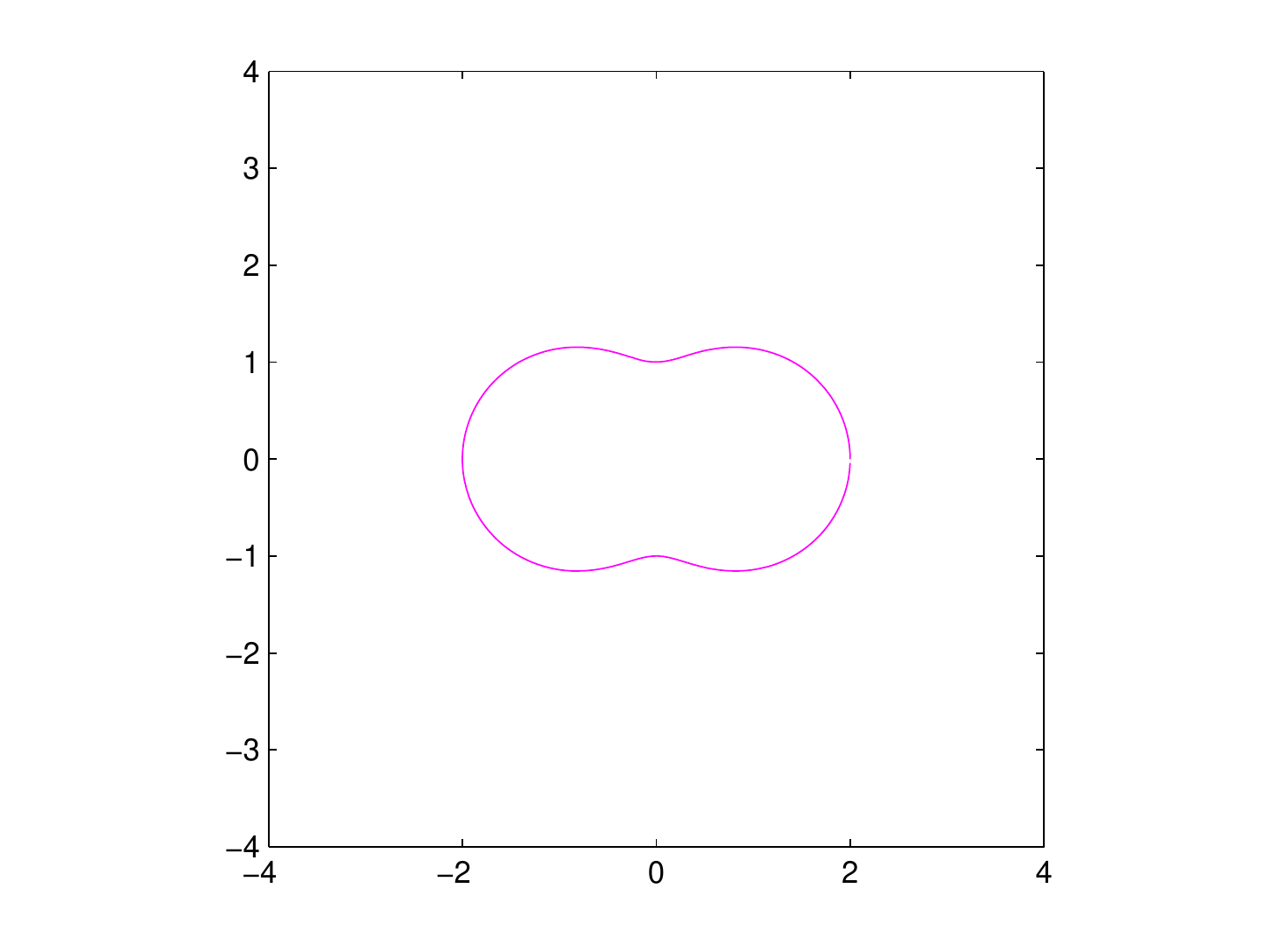}}
  \subfigure[\textbf{Pear}]{
    \includegraphics[width=1.4in]{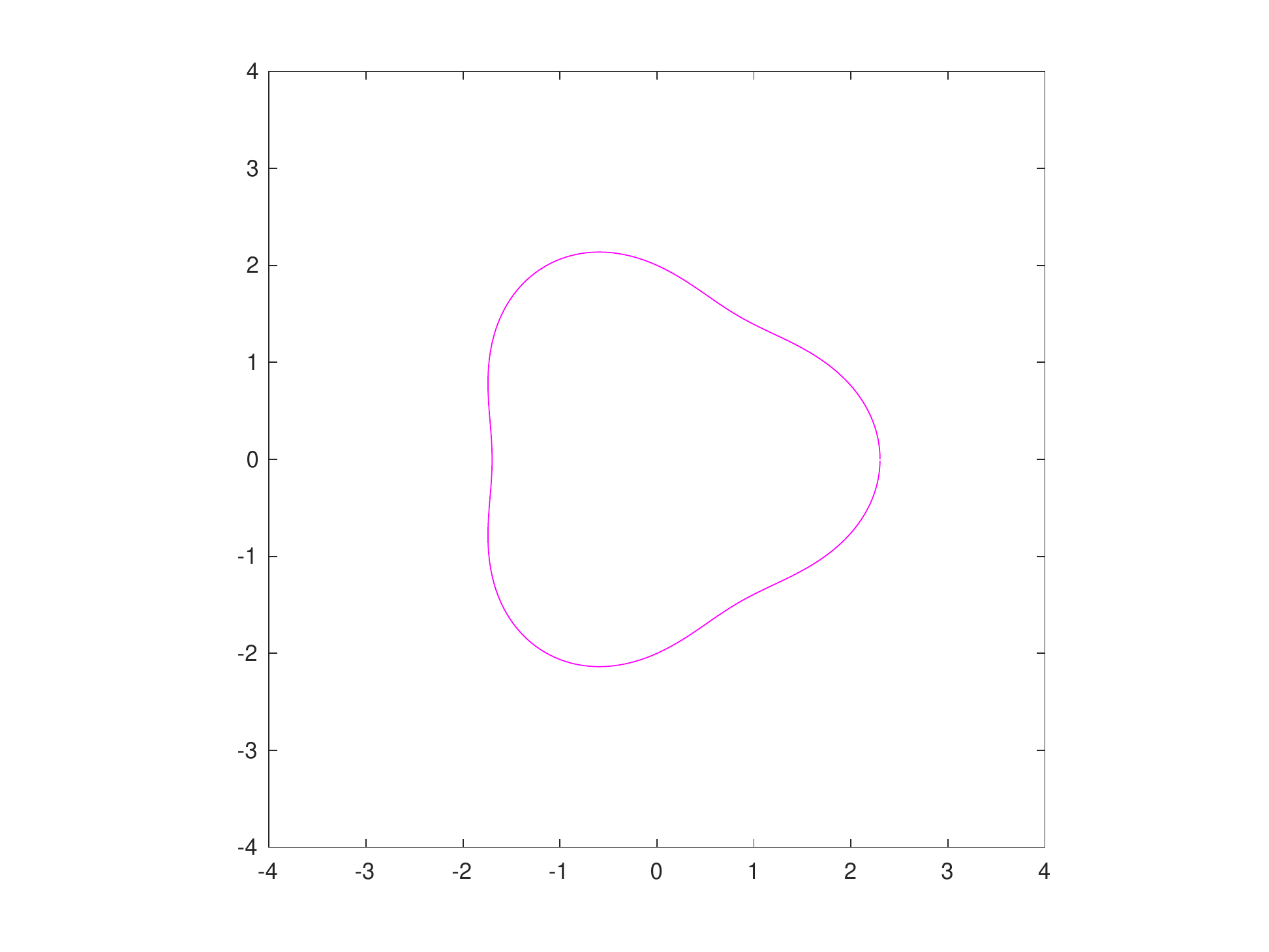}}
  \subfigure[\textbf{Circle}]{
  \includegraphics[width=1.4in]{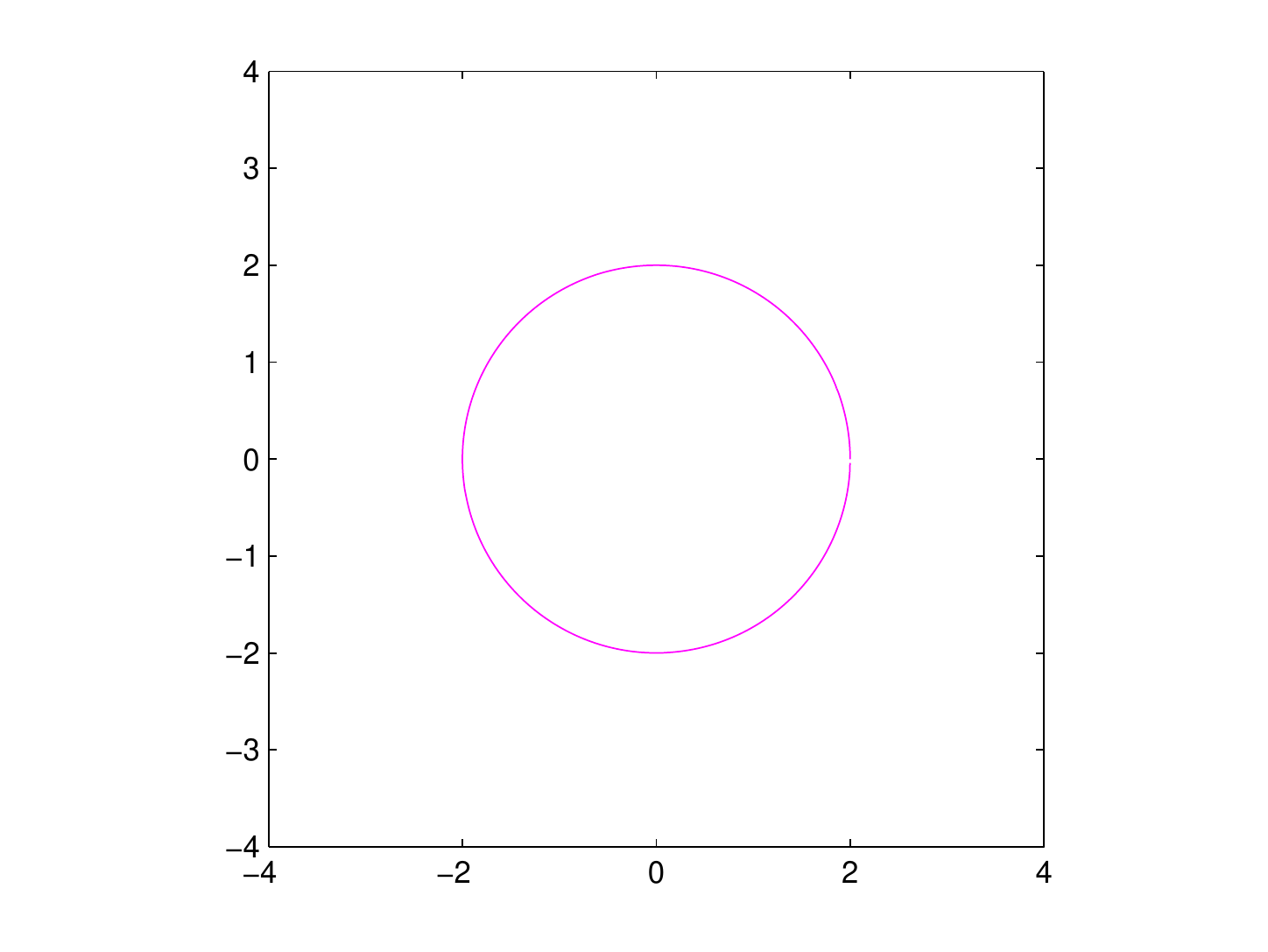}}
\caption{{\bf Different shapes to be used in the later examples.}}
\label{truedomains}
\end{figure}

Define $\theta_m:=2\pi m/N,\,m=0,1,\cdots,N-1$, let $\hth_l=(\cos\theta_l,\sin\theta_l)$ and $\hx_j=(\cos\theta_j, \sin\theta_j)$ for $j,l=0,1,\cdots, N-1$. In our simulations, we use the boundary integral
equation method to compute the far field patterns $u_{D\cup\{z_0\}}^\infty(\hx_j, \hth_l, \tau)$, $j,l=0,1,\cdots, N-1$,
for $N$ equidistantly distributed incident directions and $N$ observation directions.
We further perturb this data by random noise
\ben
\Big|u_{D\cup\{z_0\}}^{\infty,\delta}(\hx_j, \hth_l,\tau)\Big| = |u_{D\cup\{z_0\}}^{\infty}(\hx_j, \hth_l,\tau)| (1+\delta*e_{rel}), \quad j,l=0,1,\cdots, N-1,
\enn
where $e_{rel}$ is a uniformly distributed random number in the open interval $(-1,1)$.
The value of $\delta$ used in our code is the relative error level.
We also consider absolute error in {\bf Example PhaseRetrieval}. In this case, we perturb the phaseless data
\ben
\Big|u_{D\cup\{z_0\}}^{\infty,\delta}(\hx_j, \hth_l,\tau)\Big| = \max\Big\{0,\,|u_{D\cup\{z_0\}}^{\infty}(\hx_j, \hth_l,\tau)|+\delta*e_{abs}\Big\},\quad j,l=0,1,\cdots, N-1,
\enn
where $e_{abs}$ is again a uniformly distributed random number in the open interval $(-1,1)$.
Here, the value $\delta$ denotes the total error level in the measured data.

In the simulations, we use $0.05$ as the sampling space and $N=512, k=8$. If not otherwise stated,
we take $z_0=(12,12)$.

In the first four examples, we consider the indicators ${\bf I_{z_0}}$ and ${\bf I^{\Theta}_{z_0}}$ given by \eqref{Indicator0f} and \eqref{Indicator01}, respectively, with $\tau=1$.

\textbf{Example ${\bf I_{z_0}}$-Soft}. This example checks the validity of our method for scatterers with
different reference points. For simplicity, we impose Dirichlet boundary condition on the underlying
scatterer.  The scatterer is a kite with $(a,b)=(0,0)$.
Figure \ref{softkite8} shows the results with $10\%$ noise and three reference points $z_0=(2,4), z_0=(4,4)$
and $z_0=(12,12)$.
As expected, the indicator ${\bf I_{z_0}}$ takes a large value on $\pa D\cup\pa D(z_0)$, where $D(z_0)$
is the symmetric domain of $D$ about the reference point $z_0$. The symmetric domain of $z_0=(12,12)$
is outside of the sampling space. Note that $D(z_0)$ changes as the reference point $z_0$ changes,
thus it is very easy to pick the correct domain $D$ by considering the indicator ${\bf I_{z_0}}$
with different reference points, or we can just choose $z_0$ far enough.
As shown in Figures \ref{softkite8}, the left hand scatterer should be the one searched. \\

\begin{figure}[htbp]
  \centering
  \subfigure[\textbf{$z_0=(2,4)$.}]{
    \includegraphics[width=2in]{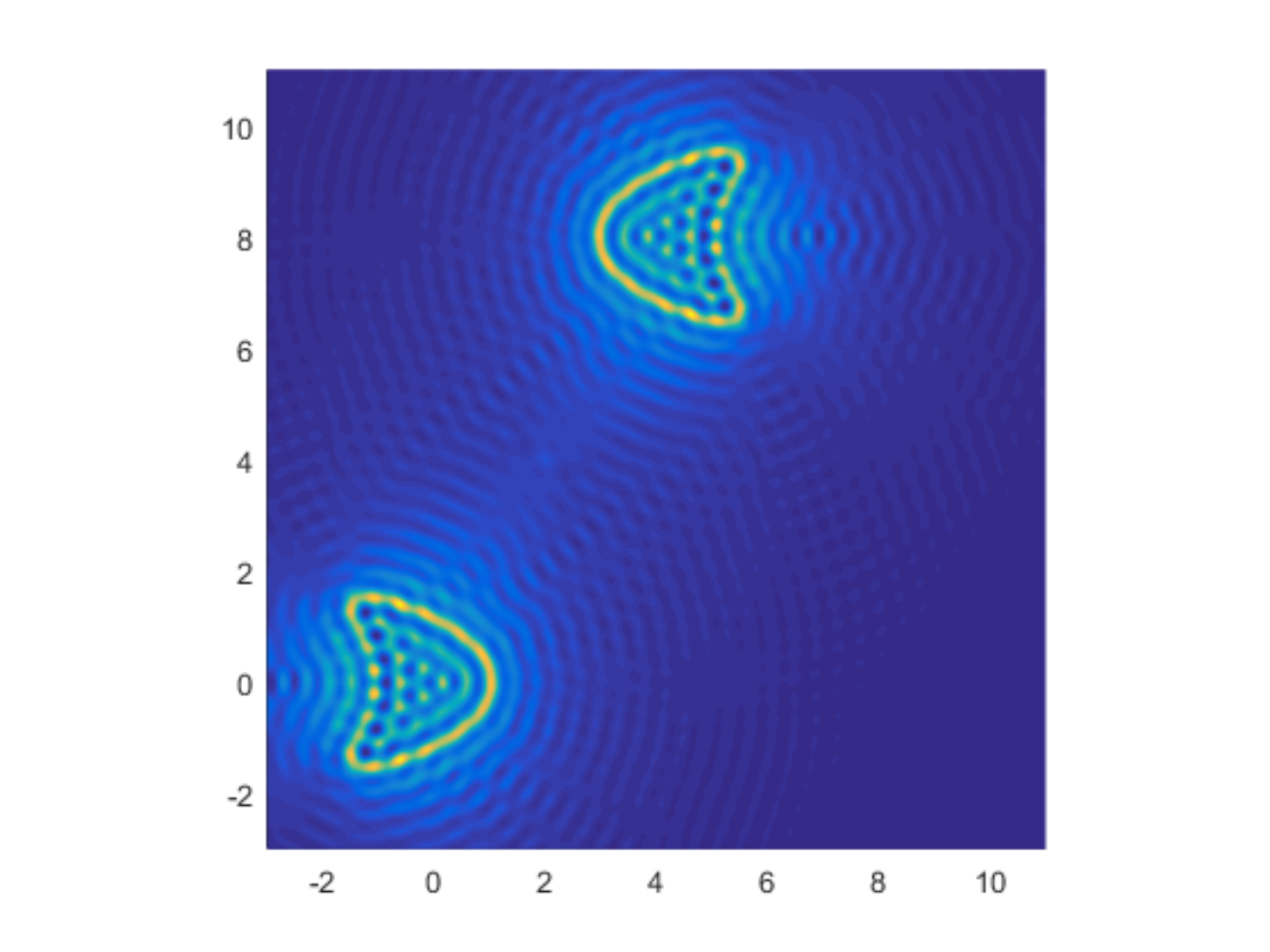}}
  \subfigure[\textbf{$z_0=(4,4)$.}]{
    \includegraphics[width=2in]{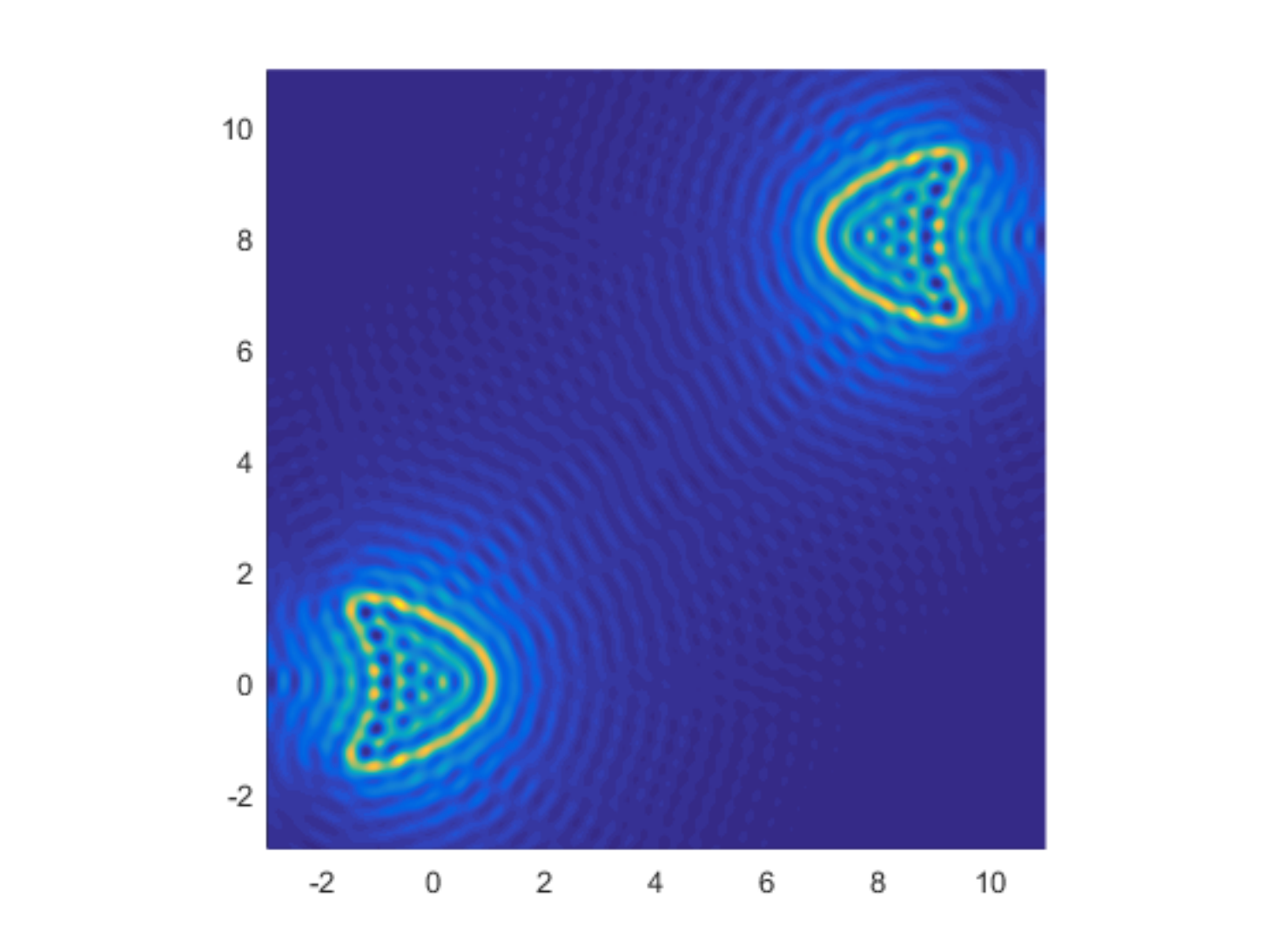}}
  \subfigure[\textbf{$z_0=(12,12)$.}]{
    \includegraphics[width=2in]{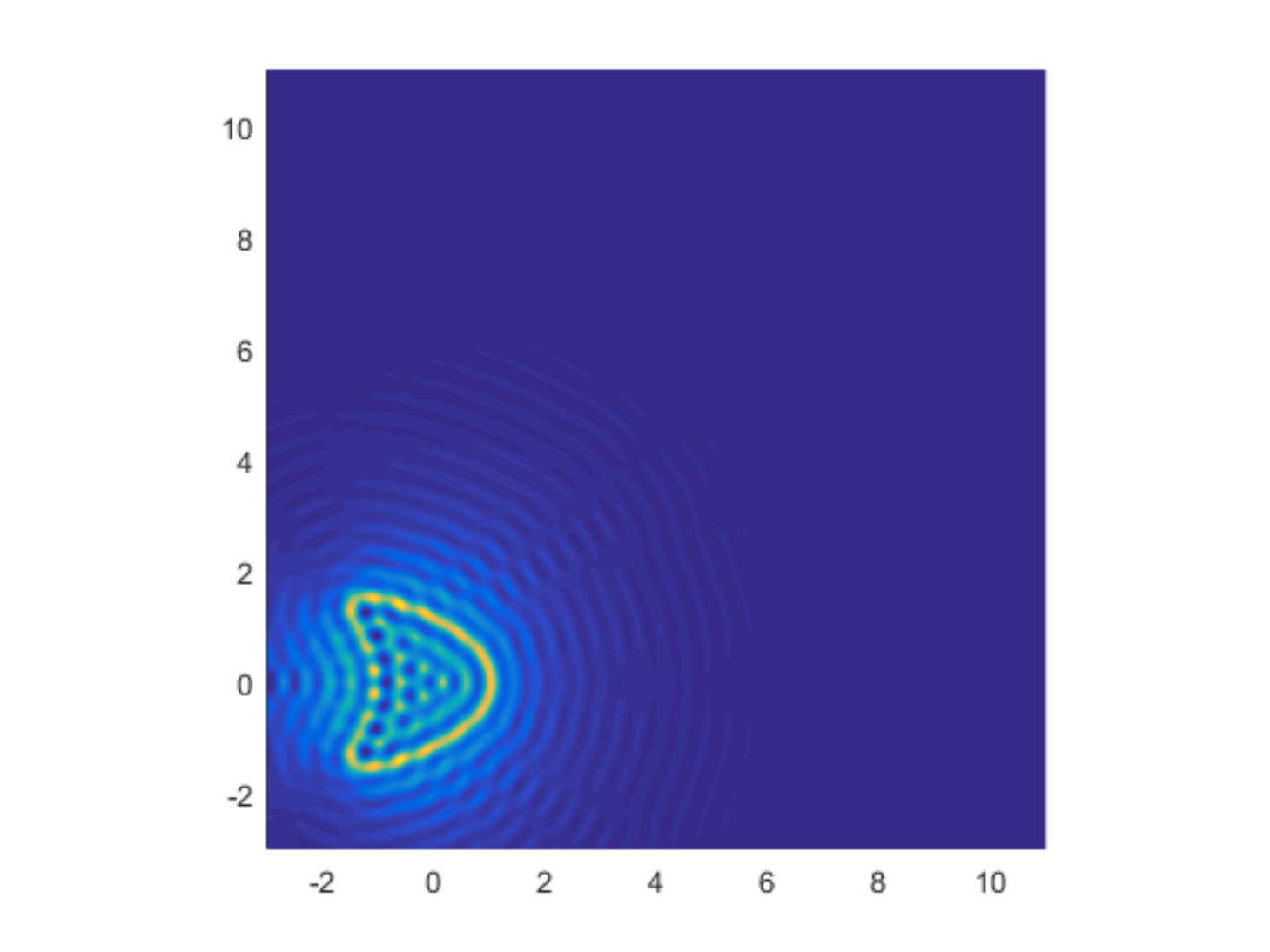}}
\caption{{\bf Example ${\bf I_{z_0}}$-Soft.}\, Reconstruction of Kite shaped domain with $10\%$
noise and different reference points.}
\label{softkite8}
\end{figure}

\textbf{Example ${\bf I_{z_0}}$-Multiple}. We consider the scattering by a scatterer with two disjoint components.
The scatterer is a combination of a sound-soft peanut shaped domain with $(a,b)=(0,0)$ and a sound-hard
kite shaped domain with $(a,b)=(6,0)$.  Figure \ref{kitepeanut8} shows the reconstructions with different noises.

\begin{figure}[htbp]
  \centering
  \subfigure[\textbf{True domain.}]{
    \includegraphics[width=2in]{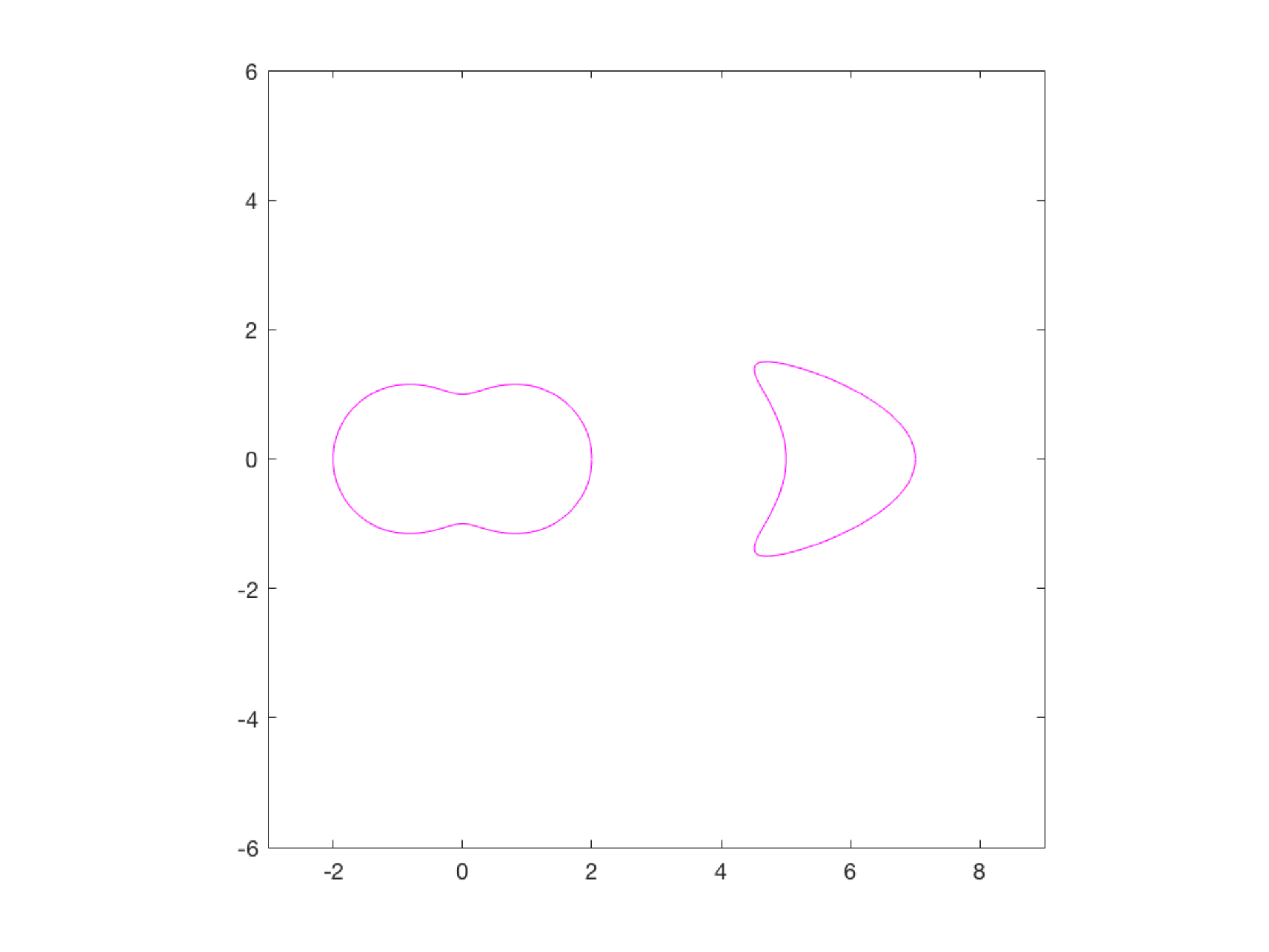}}
  \subfigure[\textbf{$10\%$noise.}]{
    \includegraphics[width=2in]{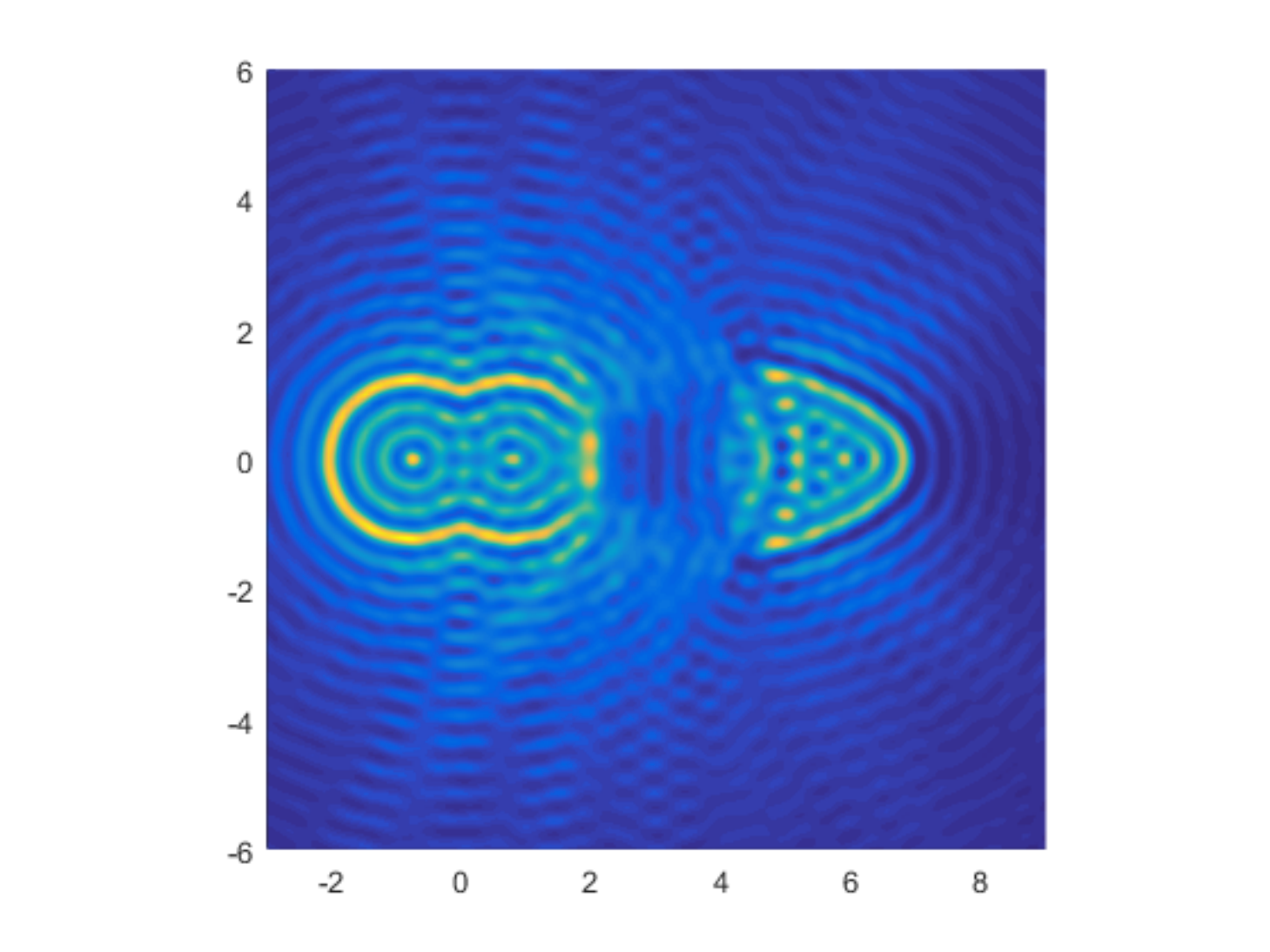}}
  \subfigure[\textbf{$30\%$noise.}]{
    \includegraphics[width=2in]{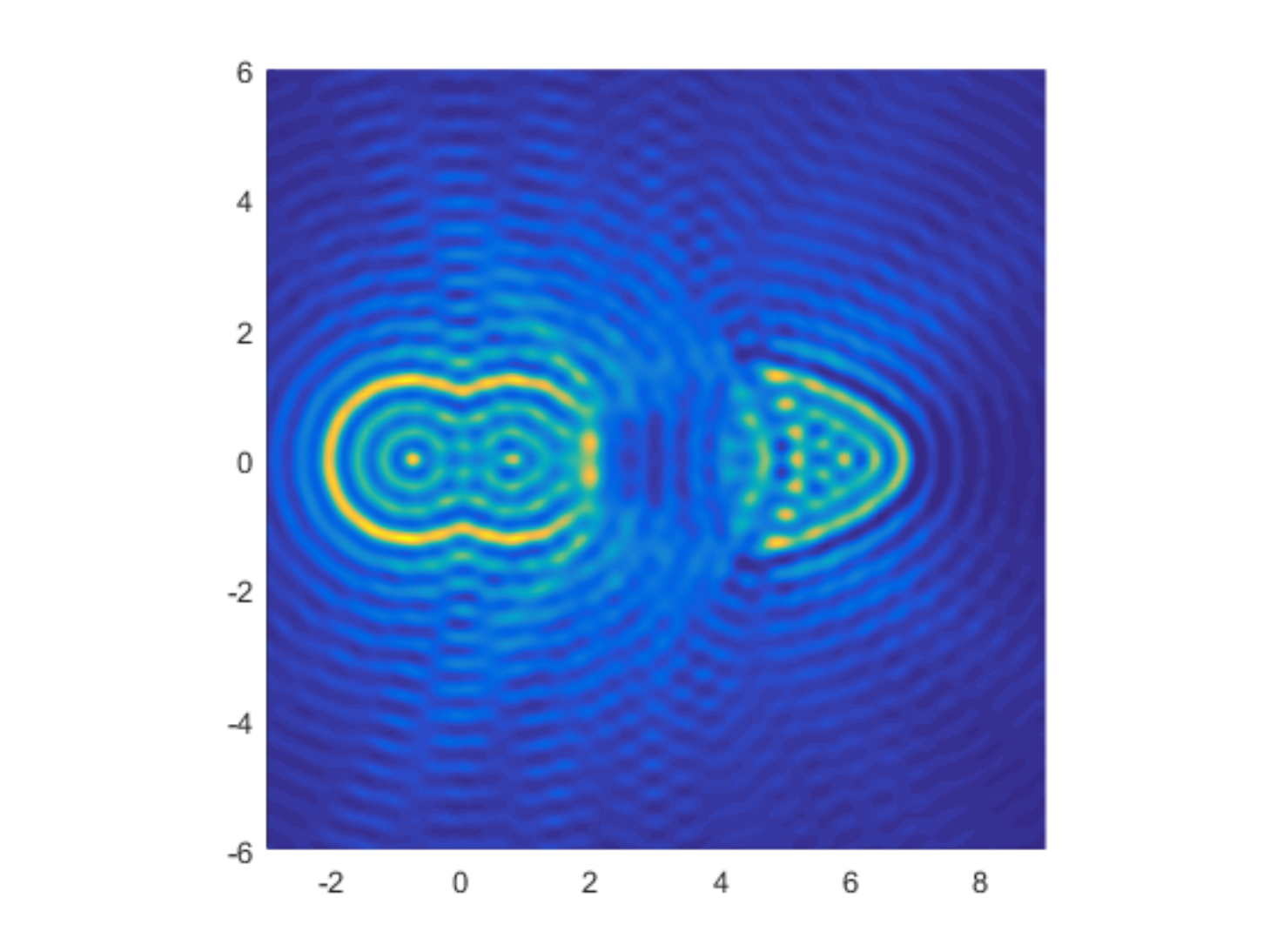}}
\caption{{\bf Example ${\bf I_{z_0}}$-Multiple.}\, Reconstruction of mixed type scatterers with different noise.}
\label{kitepeanut8}
\end{figure}

\textbf{Example ${\bf I_{z_0}}$-Multiscalar}. In this example, the underlying scatterer is a
combination of a big pear domain centered at $(0,0)$ and a mini disk with radius $r=0.1$ centered at
$(a,b)=(2,2)$. We impose Dirichlet boundary condition on both of them. The reconstructions are shown in Figure $\ref{irclesquare8}$. We observe that both parts can be reconstructed clearly. In particular, the mini disk is
also exactly located, even with $30\%$ noise.

\begin{figure}[htbp]
  \centering
  \subfigure[\textbf{True domain.}]{
    \includegraphics[width=2in]{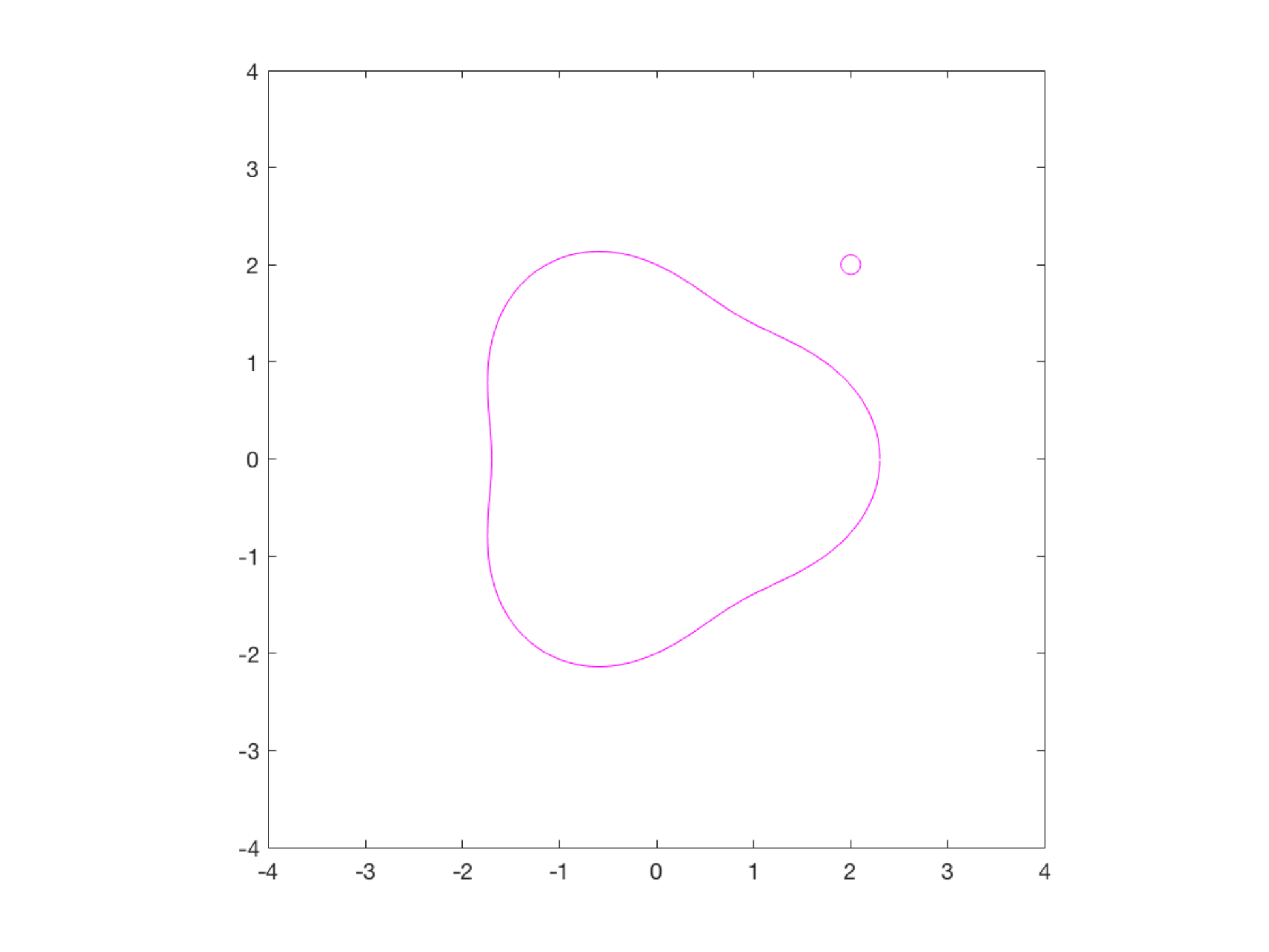}}
  \subfigure[\textbf{$10\%$noise.}]{
    \includegraphics[width=2in]{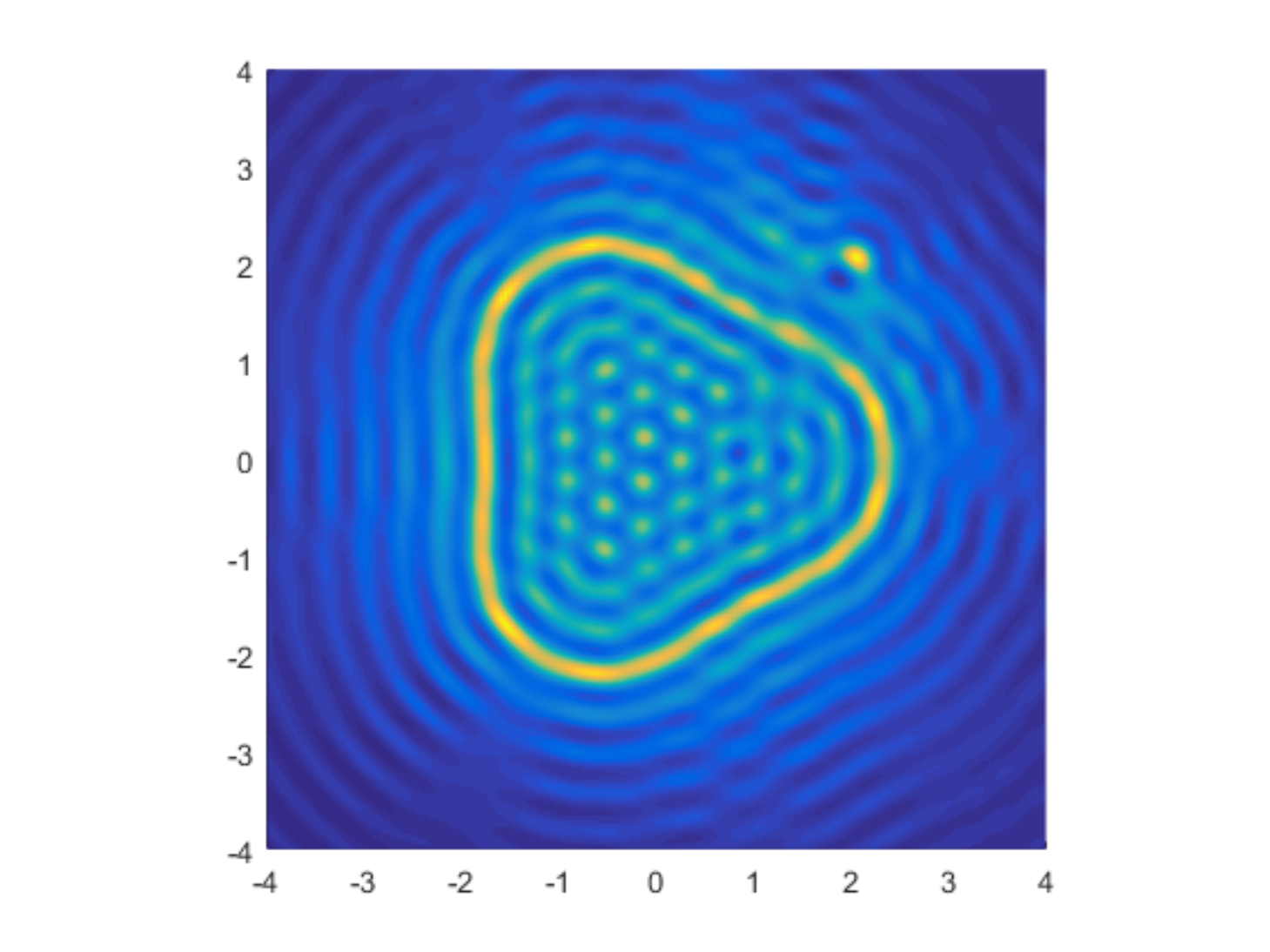}}
  \subfigure[\textbf{$30\%$noise.}]{
    \includegraphics[width=2in]{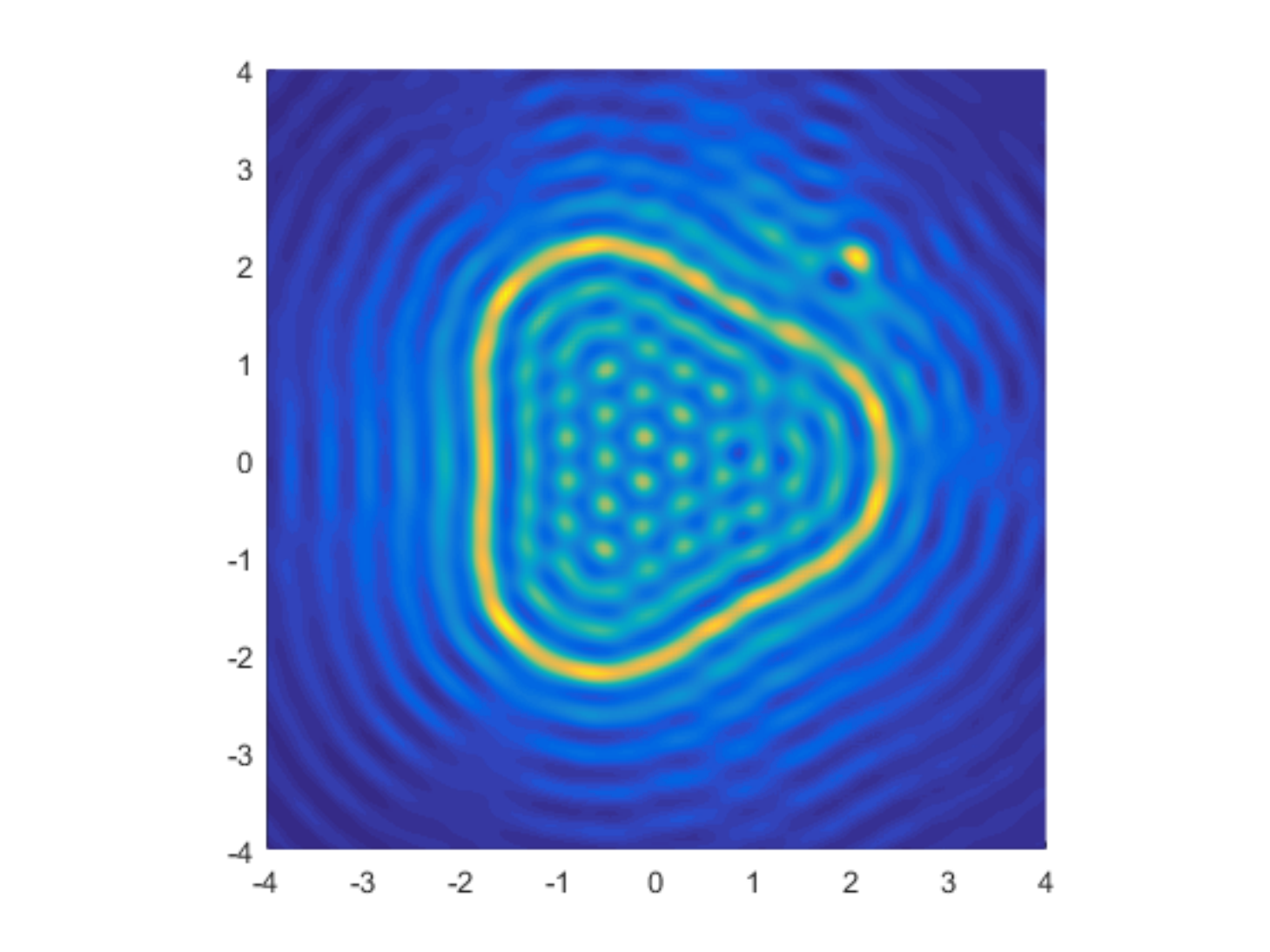}}
\caption{{\bf Example ${\bf I_{z_0}}$-Multiscalar.}\, Reconstruction of multiscalar scatterers
with different noise.}
\label{irclesquare8}
\end{figure}

\textbf{Example ${\bf I^{\Theta}_{z_0}}$-Small}. In this example, the scatterer is a combination of two
mini disks, one with radius $0.05$ centered at
$(a,b)=(3,3)$ and the other with radius $0.15$ centered at $(a,b)=(1,1)$. We impose Dirichlet boundary
condition on the smaller disk and Neumann boundary condition on the bigger one. Figure \ref{iTheta} shows
the reconstructions by ${\bf I^{\Theta}_{z_0}}$ with the same reference points as in
the \textbf{Example ${\bf I_{z_0}}$-Soft}.

\begin{figure}[htbp]
  \centering
  \subfigure[\textbf{$z_0=(2,4)$.}]{
    \includegraphics[width=2in]{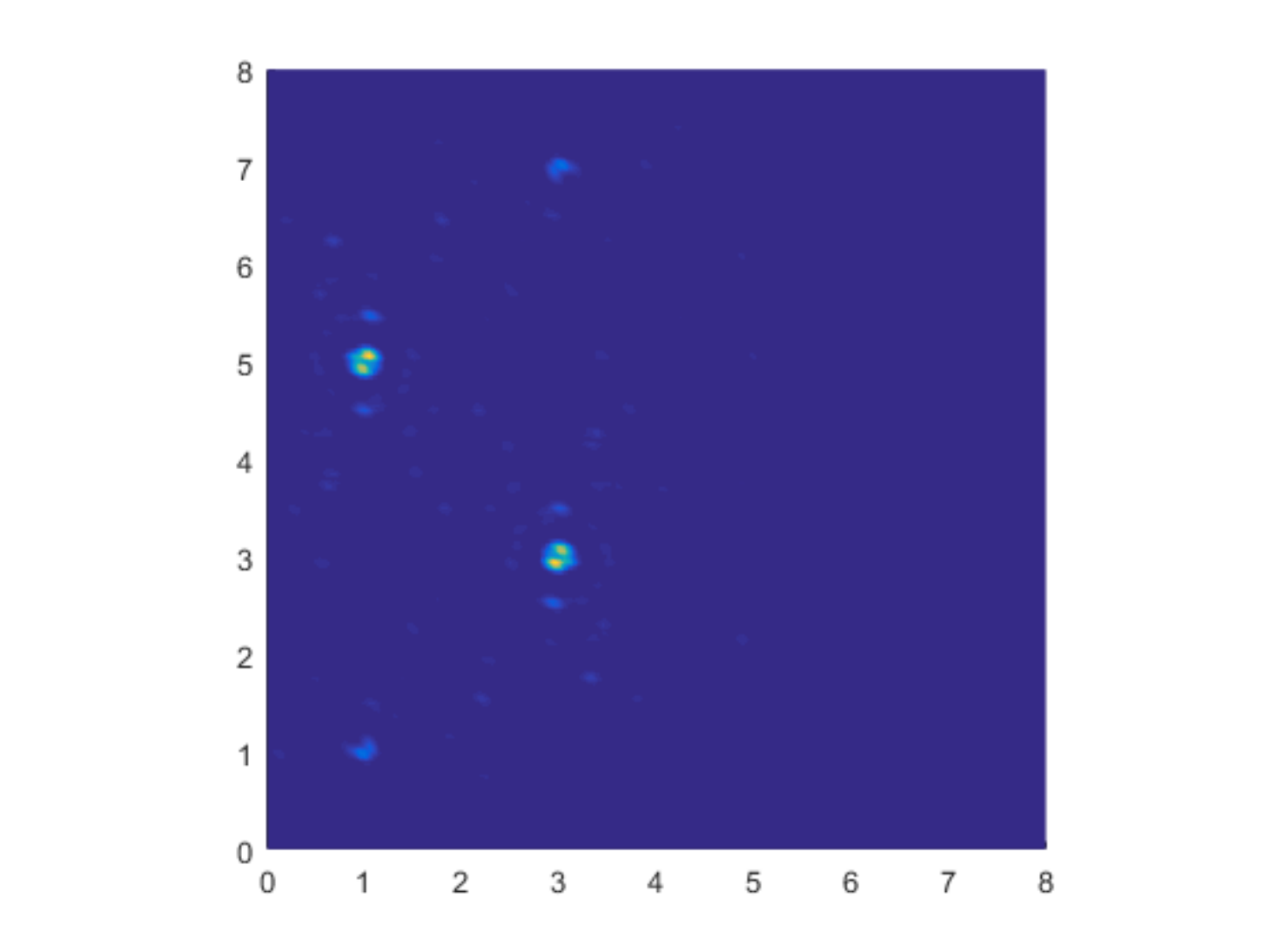}}
  \subfigure[\textbf{$z_0=(4,4)$.}]{
    \includegraphics[width=2in]{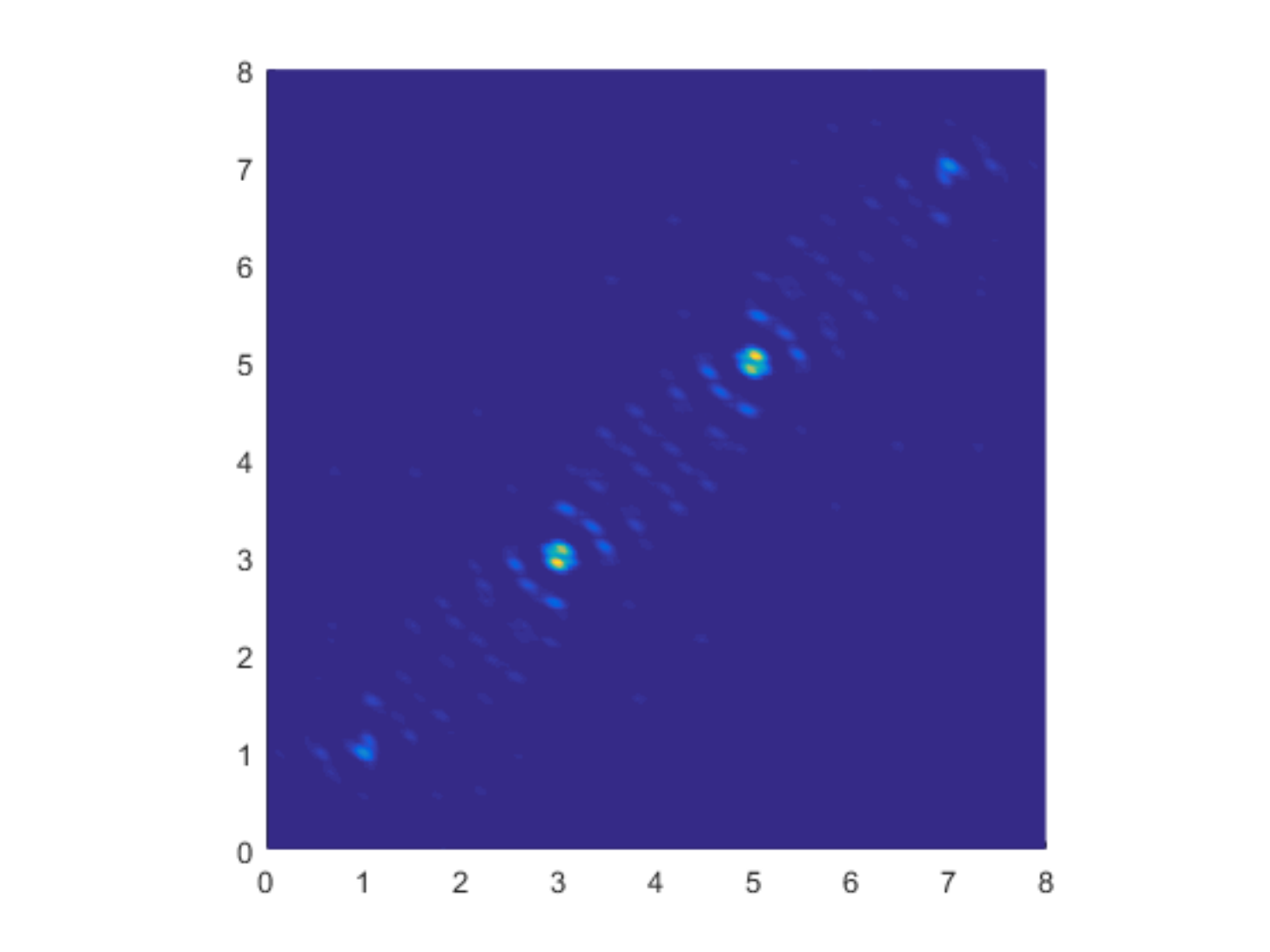}}
  \subfigure[\textbf{$z_0=(12,12)$.}]{
    \includegraphics[width=2in]{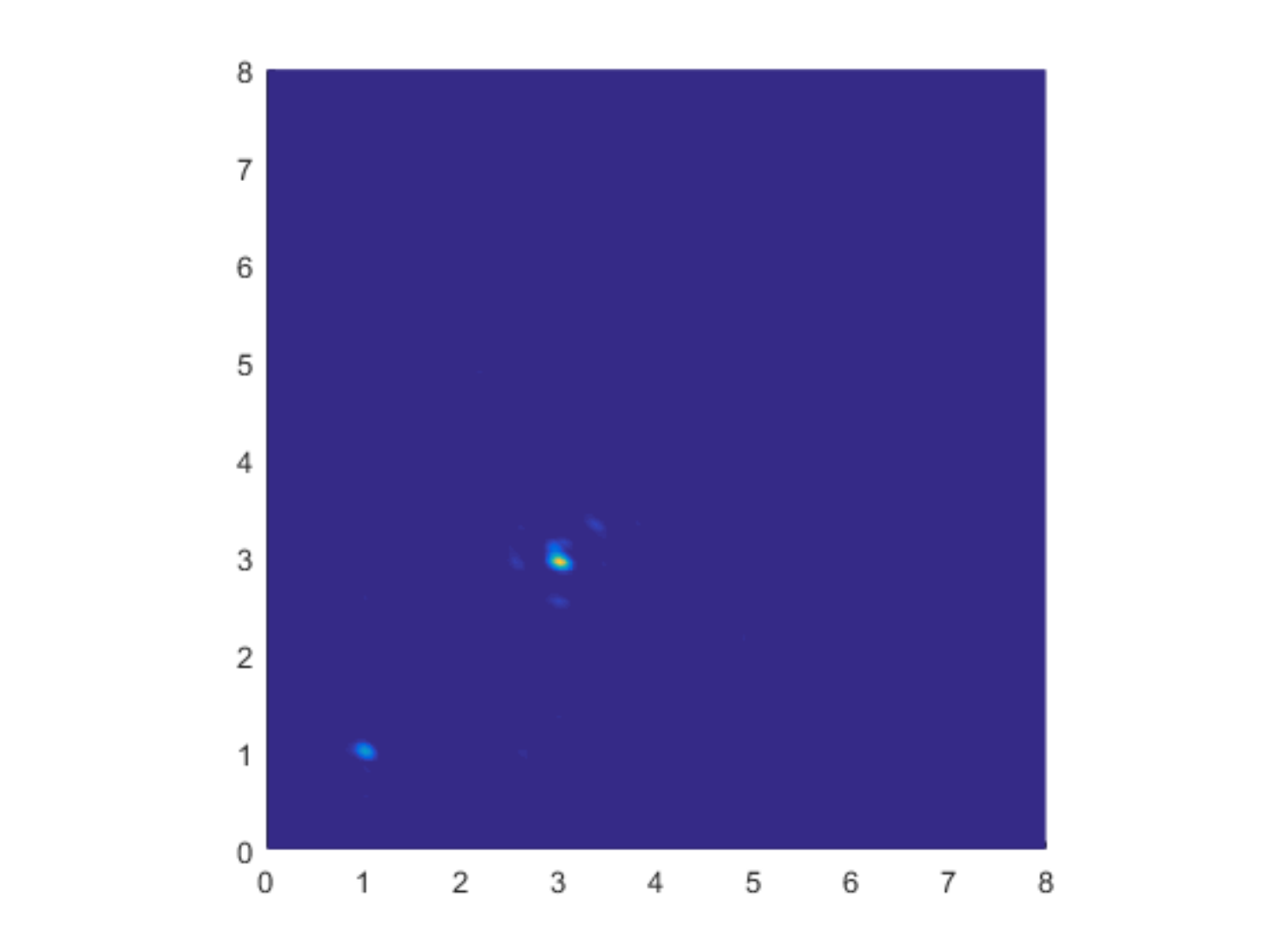}}
\caption{{\bf Example ${\bf I^{\Theta}_{z_0}}$-Small.}\, Reconstruction of two small disks
by using ${\bf I^{\Theta}_{z_0}}$ with $10\%$ noise at different reference points.
Here, $\Theta:=\{(1,0), (0,1), (-1,0), (0,-1)\}$.}
\label{iTheta}
\end{figure}

In the next example, we consider the effectiveness and robustness of the novel phase retrieval proposed
in Section \ref{subsec-phaseretrieval}.
After this, we check the validity of the {\bf Scatterer Reconstruction Scheme Two}.
In the following examples, we take $\tau=-1,1,i$.\\

\textbf{Example PhaseRetrieval}.
This example is designed to check the phase retrieval scheme proposed in Section \ref{subsec-phaseretrieval}.
The underlying scatterer is chosen to be a kite shaped domain.
For comparison, we consider the real part of far field pattern at a fixed incident direction $\hth=(1,0)$.
Figure \ref{phaseretrieval0} shows the results without measurement noise by using three different reference
points $(2,2), \,(3,3)$ and $(4,4)$. In particular, the reference point
$(2,2)$ is very close to the kite shaped domain. However, Figure \ref{phaseretrieval0}(a) shows that the
multiple scattering is very week. Of course, Figures \ref{phaseretrieval0}(b)-(c) show that the interaction
between the reference point and the kite shaped domain decreases as the reference point away from the target.
Figures \ref{phaseretrieval1}-\ref{phaseretrieval2} show the results with relative error and absolute error
considered, respectively.
We find that our phase retrieval scheme is quite robust with respect to noise. This also verifies the theory
provided in Theorem \ref{phaseretrieval-stability}.\\

\begin{figure}[htbp]
   \centering
   \subfigure[\textbf{$z_0=(2,2)$.}]{
     \includegraphics[height=1.5in,width=2in]{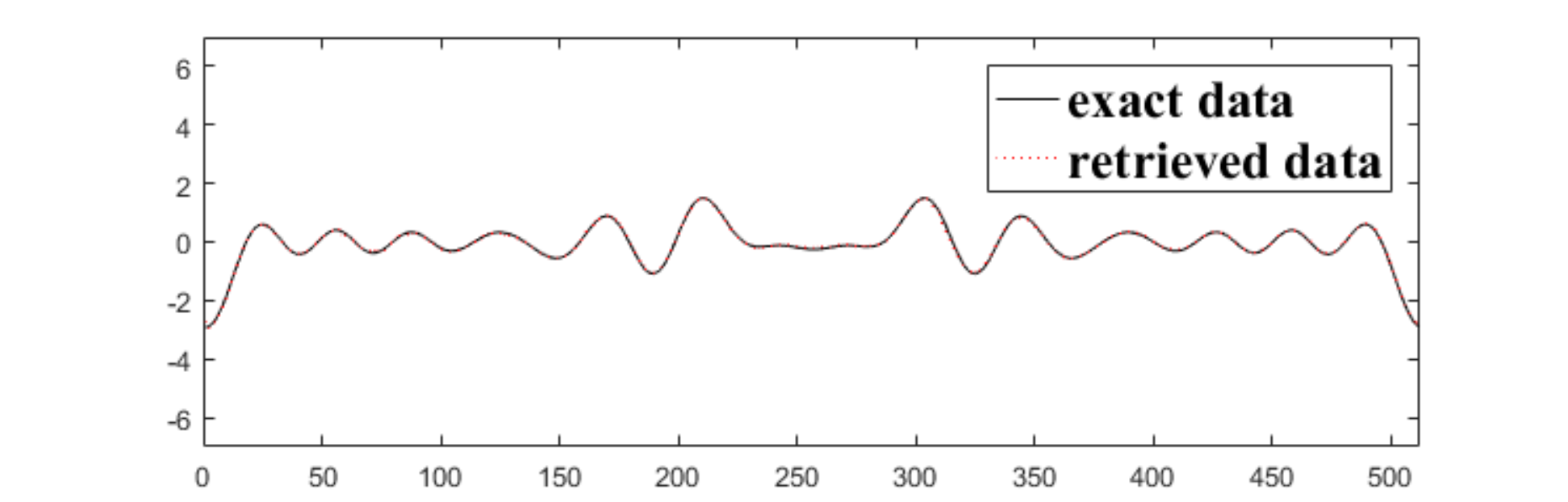}}
   \subfigure[\textbf{$z_0=(3,3)$.}]{
     \includegraphics[height=1.5in,width=2in]{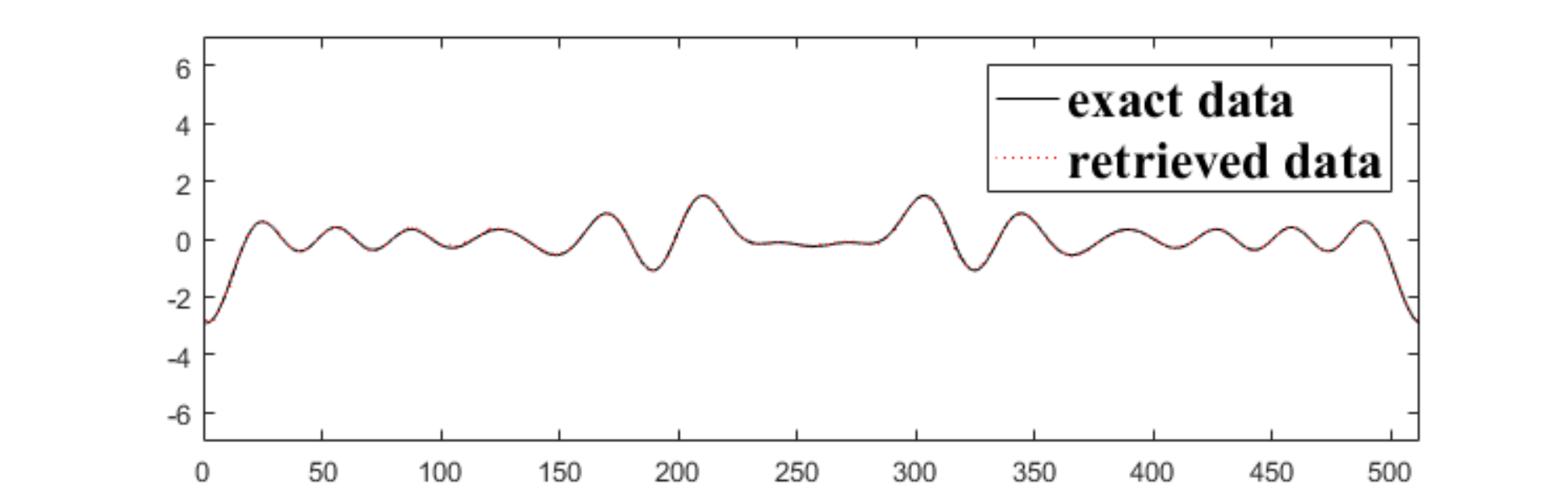}}
   \subfigure[\textbf{$z_0=(4,4)$.}]{
     \includegraphics[height=1.5in,width=2in]{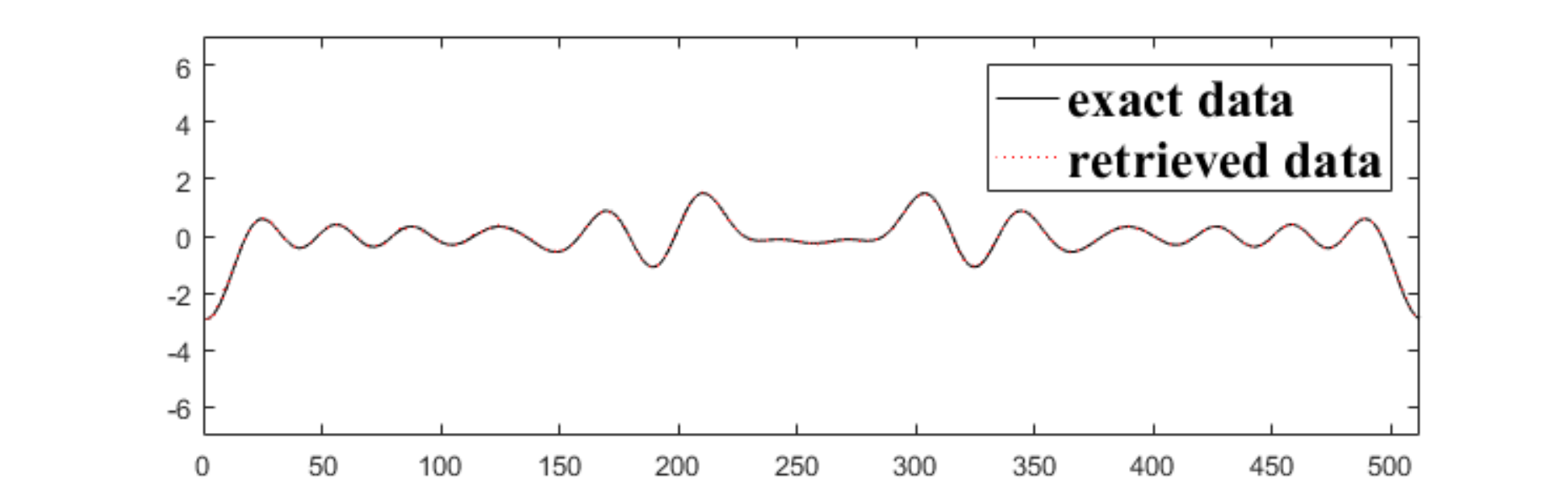}}
 \caption{{\bf Example PhaseRetrieval.}\, Phase retrieval for the real part of the far field
 pattern without error at a fixed incident direction $\hth=(1,0)$ using different reference points.}
 \label{phaseretrieval0}
 \end{figure}

\begin{figure}[htbp]
   \centering
   \subfigure[\textbf{without noise.}]{
     \includegraphics[height=1.5in,width=2in]{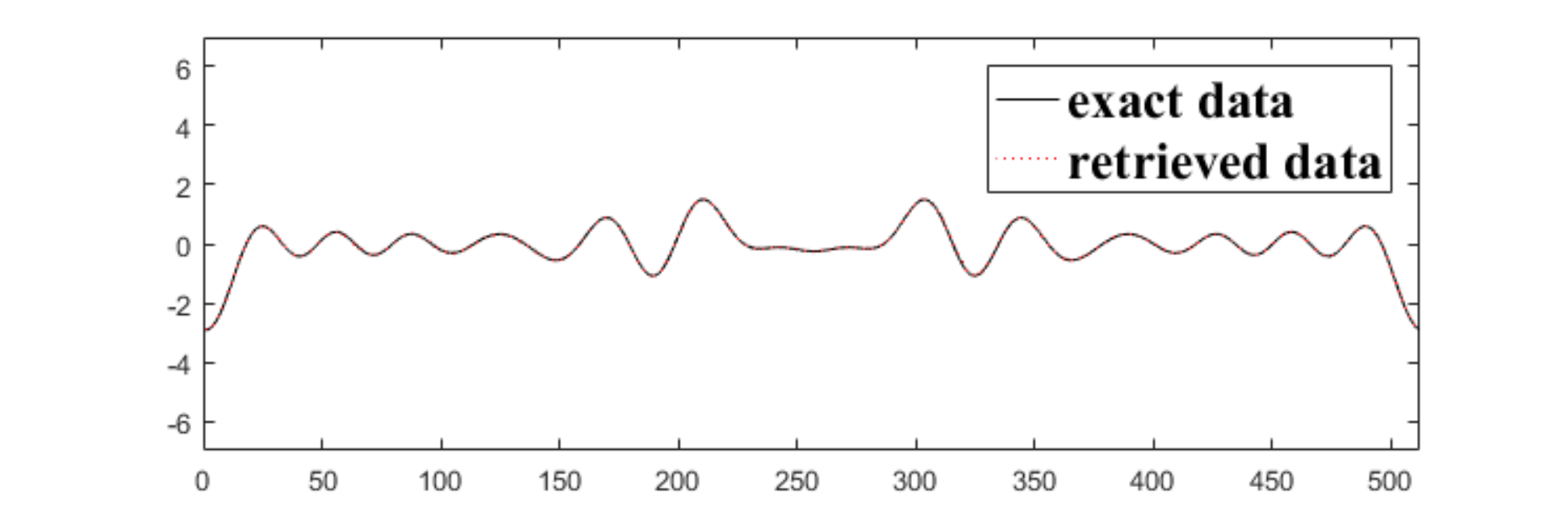}}
   \subfigure[\textbf{$10\%$ noise.}]{
     \includegraphics[height=1.5in,width=2in]{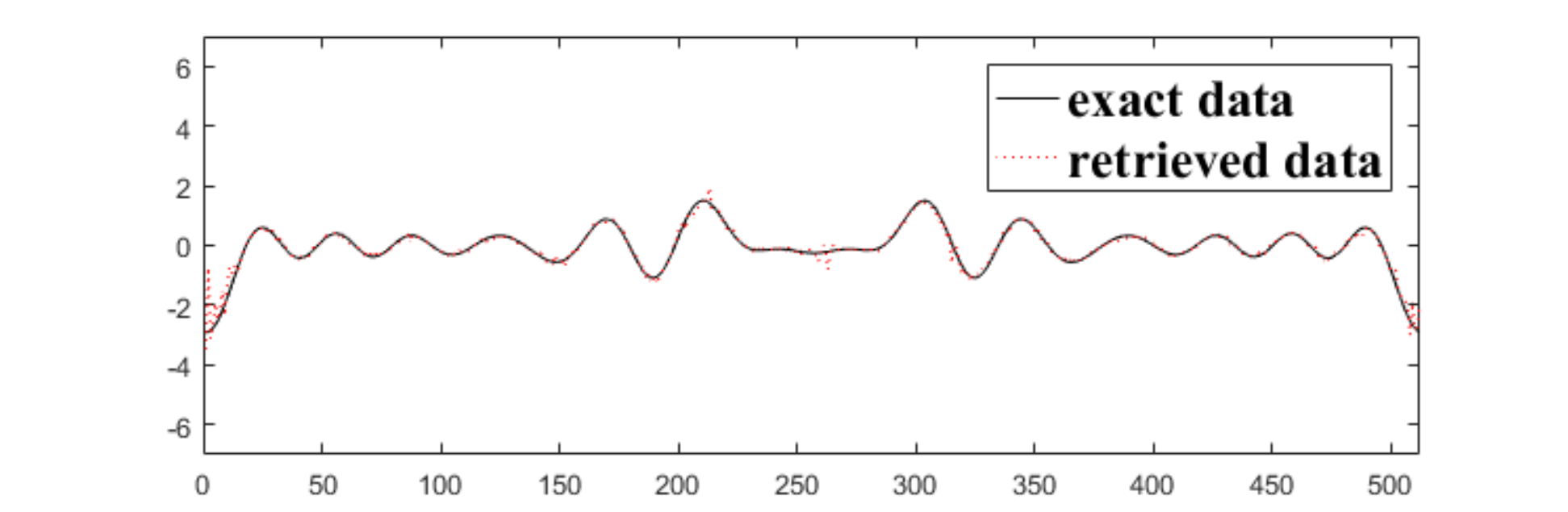}}
   \subfigure[\textbf{$30\%$ noise.}]{
     \includegraphics[height=1.5in,width=2in]{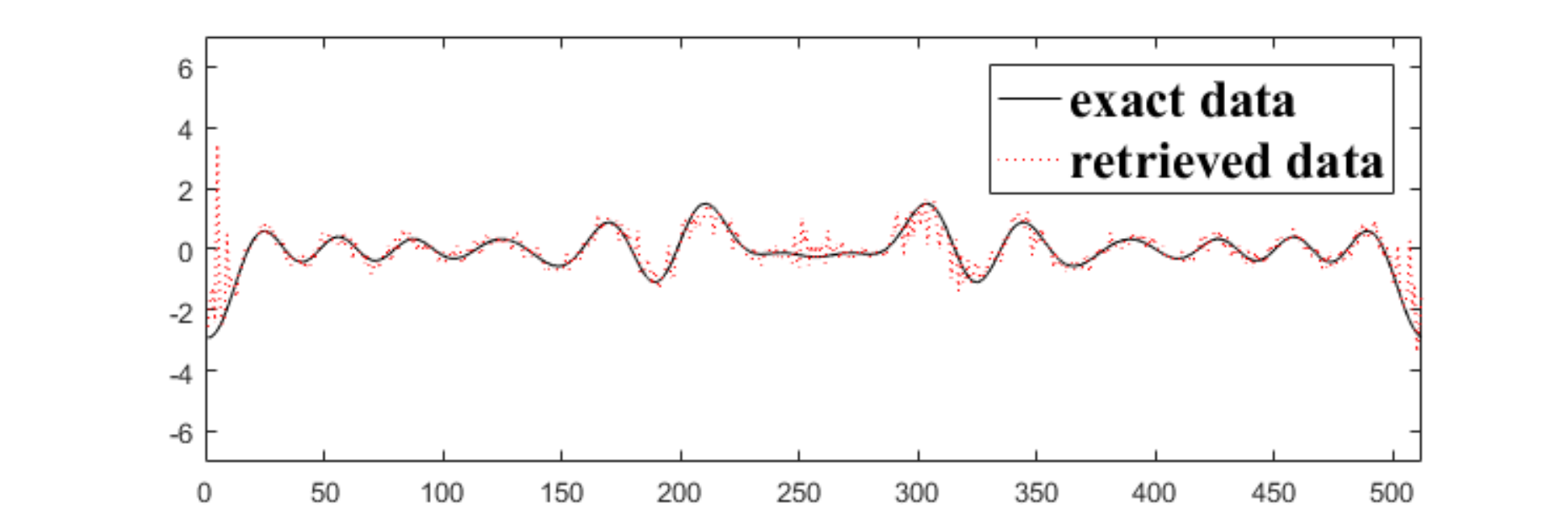}}
 \caption{{\bf Example PhaseRetrieval.}\, Phase retrieval for the real part of the far field
 pattern with relative error at a fixed incident direction $\hth=(1,0)$.}
 \label{phaseretrieval1}
 \end{figure}

 \begin{figure}[htbp]
   \centering
   \subfigure[\textbf{without noise.}]{
     \includegraphics[height=1.5in,width=2in]{pic/FRkitek8noise0n512z1212softd10.pdf}}
   \subfigure[\textbf{0.1 noise.}]{
     \includegraphics[height=1.5in,width=2in]{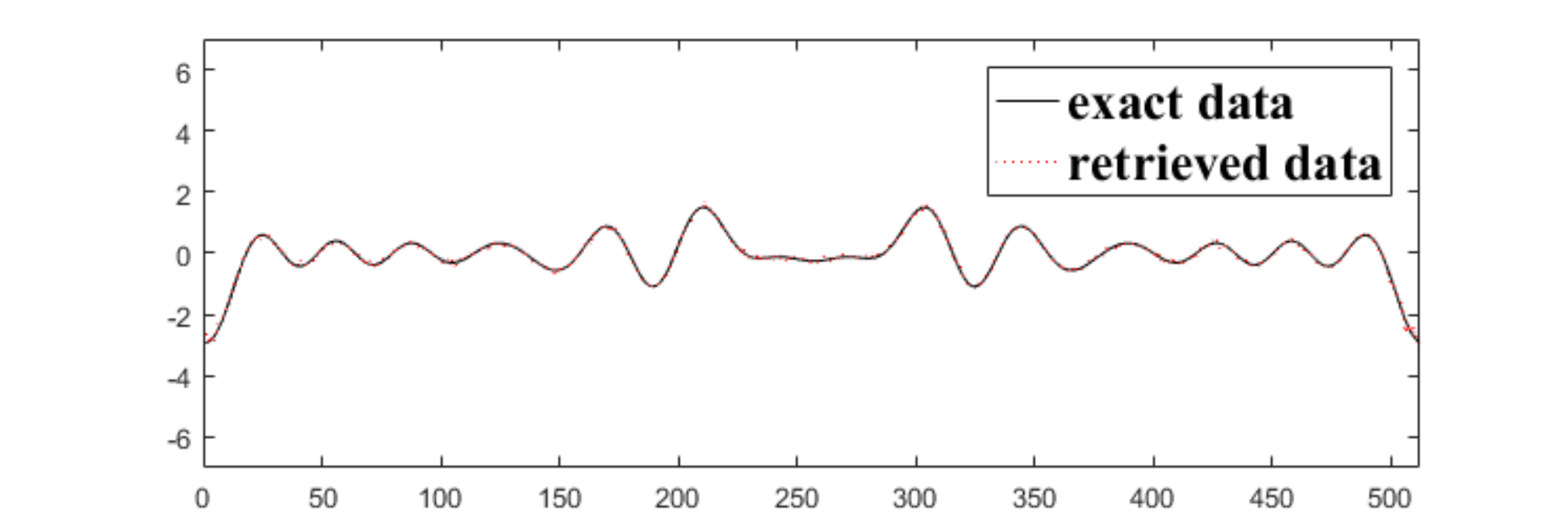}}
   \subfigure[\textbf{0.3 noise.}]{
     \includegraphics[height=1.5in,width=2in]{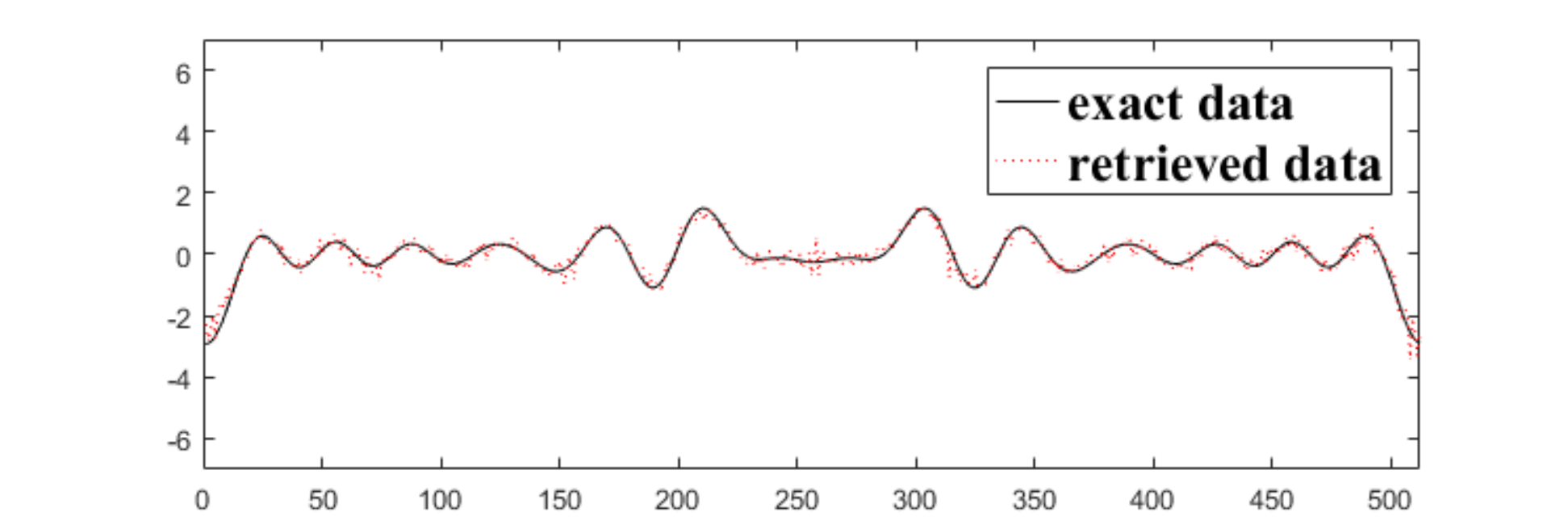}}
 \caption{{\bf Example PhaseRetrieval.}\, Phase retrieval for the real part of the far field
 pattern with absolute error at a fixed incident direction $\hth=(1,0)$.}
 \label{phaseretrieval2}
 \end{figure}

\textbf{Example ${\bf I_2}$-Soft}.  The scatterer is the same as the \textbf{Example ${\bf I_{z_0}}$-Soft}.
For comparisons, we choose the same reference points $z_0=(2,4),(4,4),(12,12)$. Figure \ref{2kite8} gives
the results with $10\%$ noise.
Different to the \textbf{Example ${\bf I_{z_0}}$-Soft}, no false domain appears in the reconstructions.

\begin{figure}[htbp]
  \centering
  \subfigure[\textbf{$z_0=(2,4)$.}]{
    \includegraphics[width=2in]{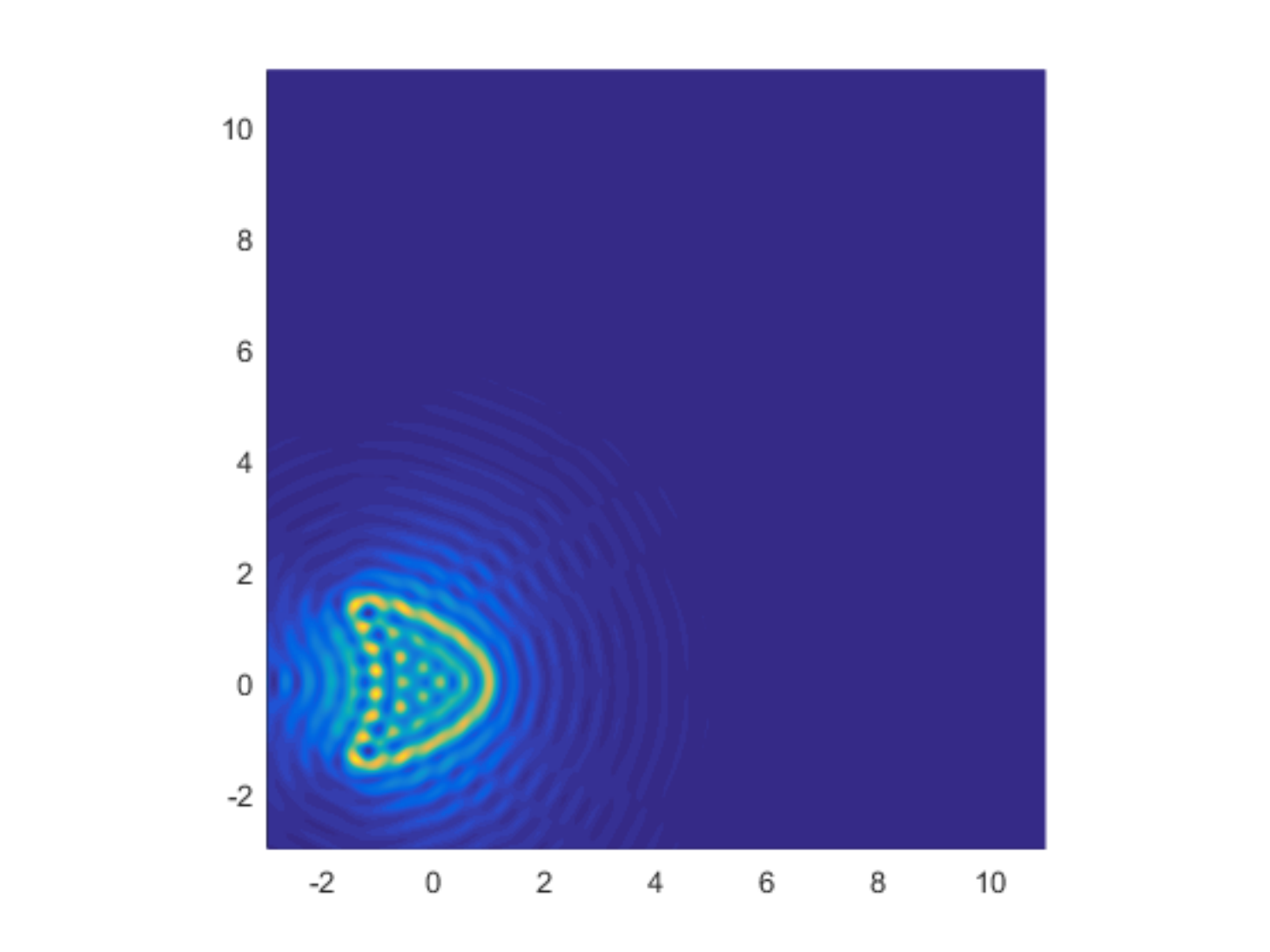}}
  \subfigure[\textbf{$z_0=(4,4)$.}]{
    \includegraphics[width=2in]{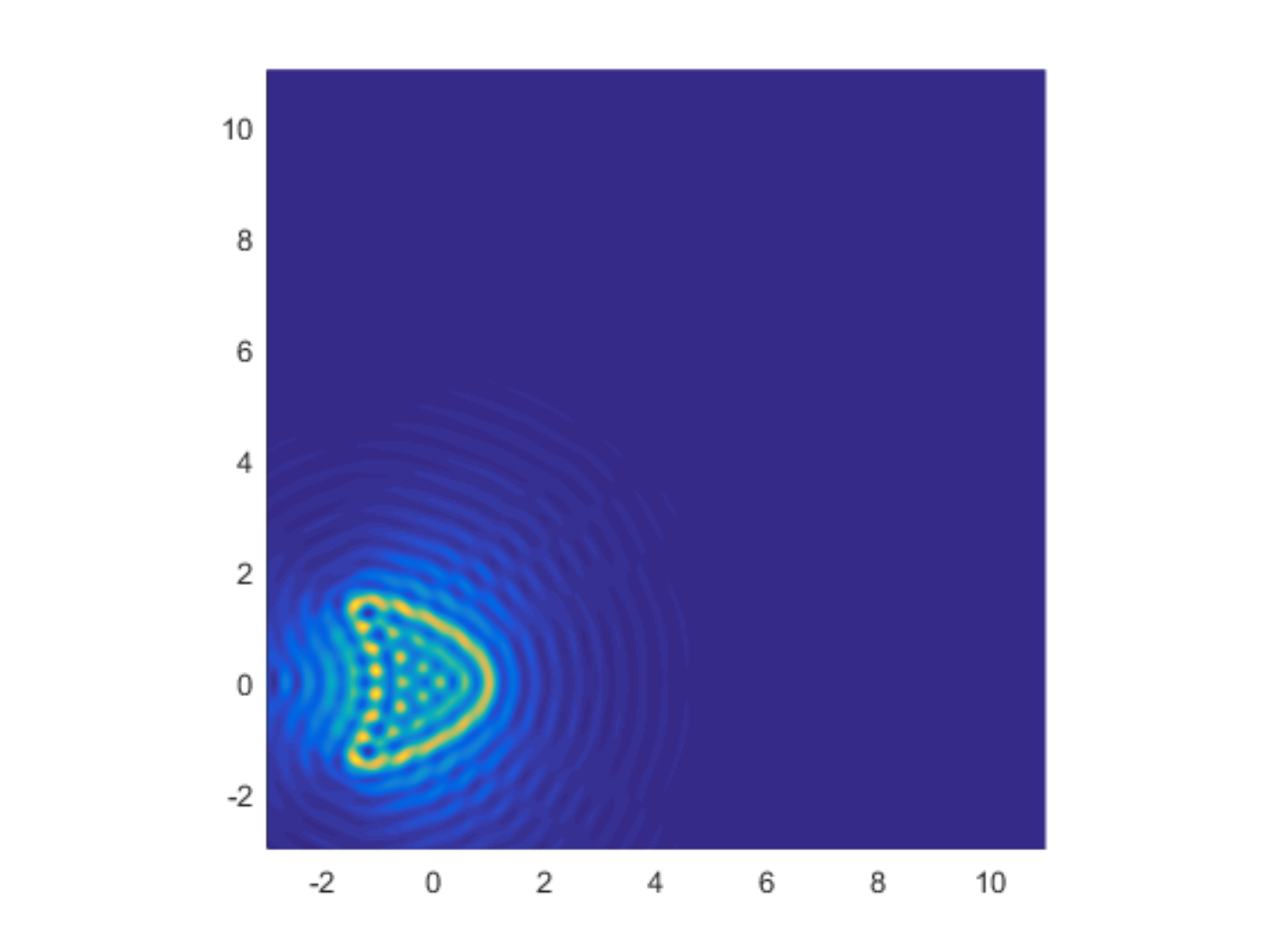}}
  \subfigure[\textbf{$z_0=(12,12)$.}]{
    \includegraphics[width=2in]{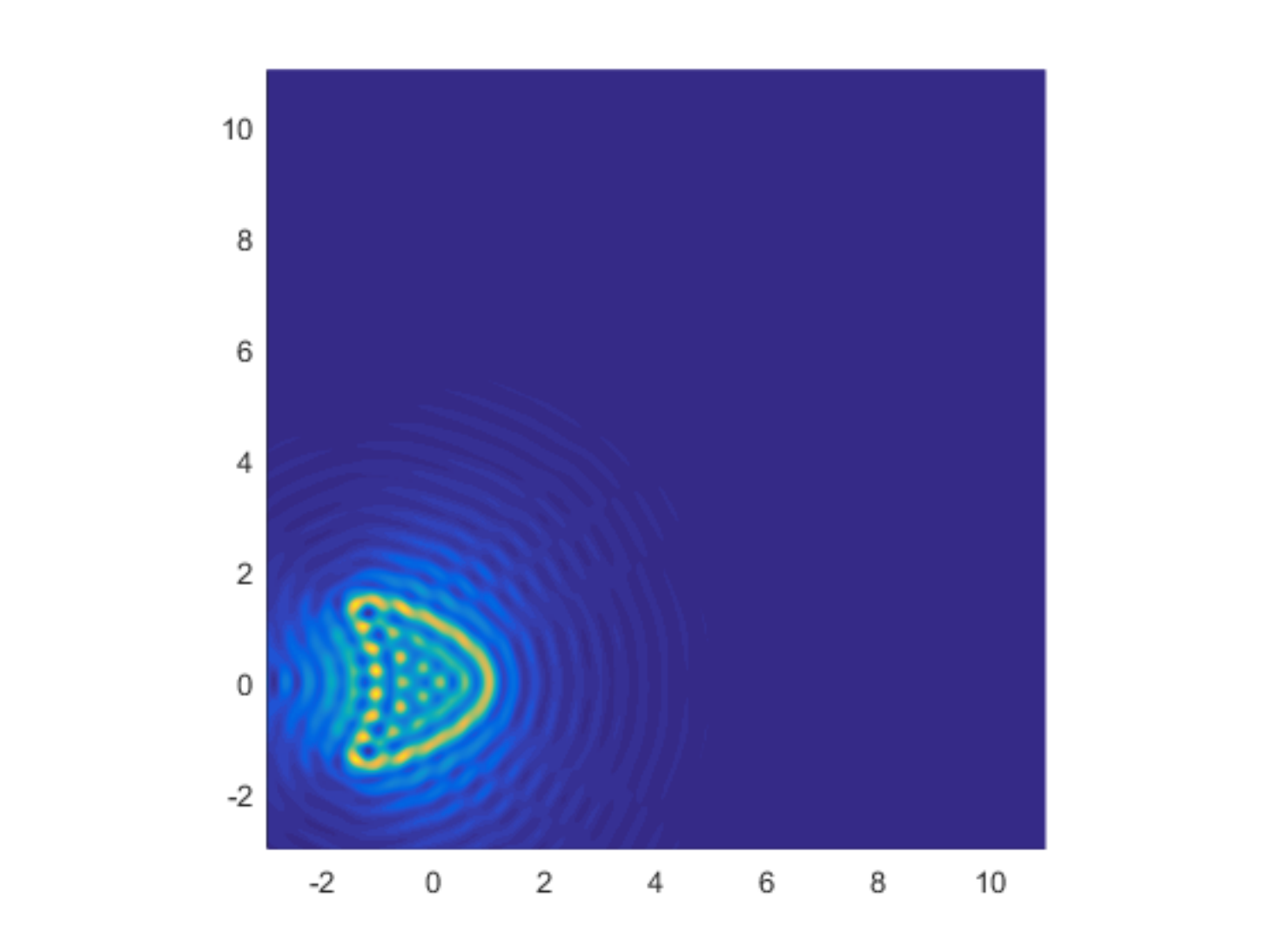}}
\caption{{\bf Example ${\bf I_2}$-Soft.}\, Reconstruction of kite shaped domain with $10\%$ noise and
different reference points.}
\label{2kite8}
\end{figure}

\textbf{Example ${\bf I_2}$-Multiple}.
The scatterer is the same as the \textbf{Example ${\bf I_{z_0}}$-Multiple}. Figure \ref{2kitepeanut8}
gives the results with $10\%, 30\%$ noise.

\begin{figure}[htbp]
  \centering
 \subfigure[\textbf{True domain.}]{
    \includegraphics[width=2in]{pic/peanutkitesofthard.pdf}}
  \subfigure[\textbf{$10\%$noise.}]{
    \includegraphics[width=2in]{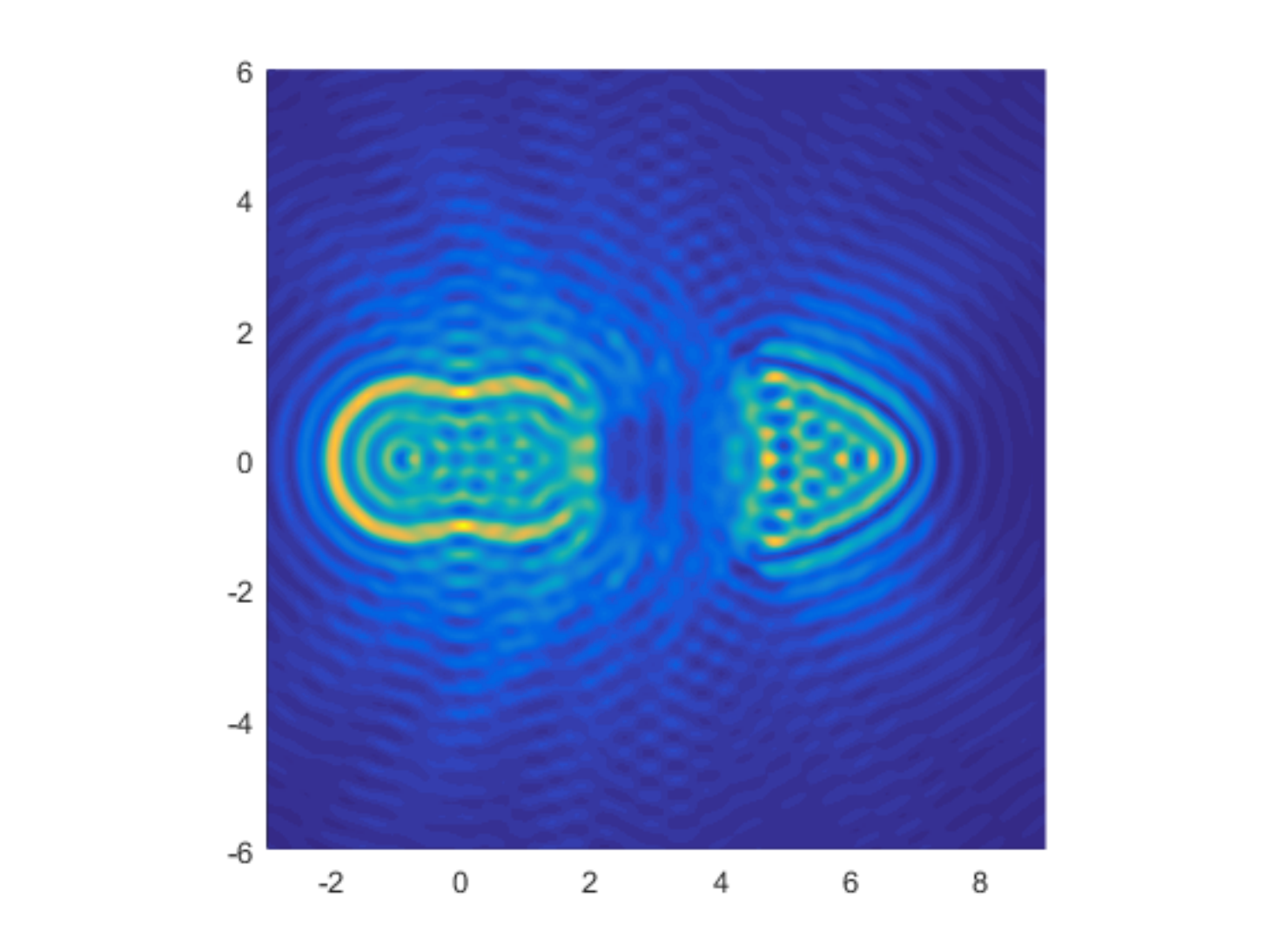}}
  \subfigure[\textbf{$30\%$noise.}]{
    \includegraphics[width=2in]{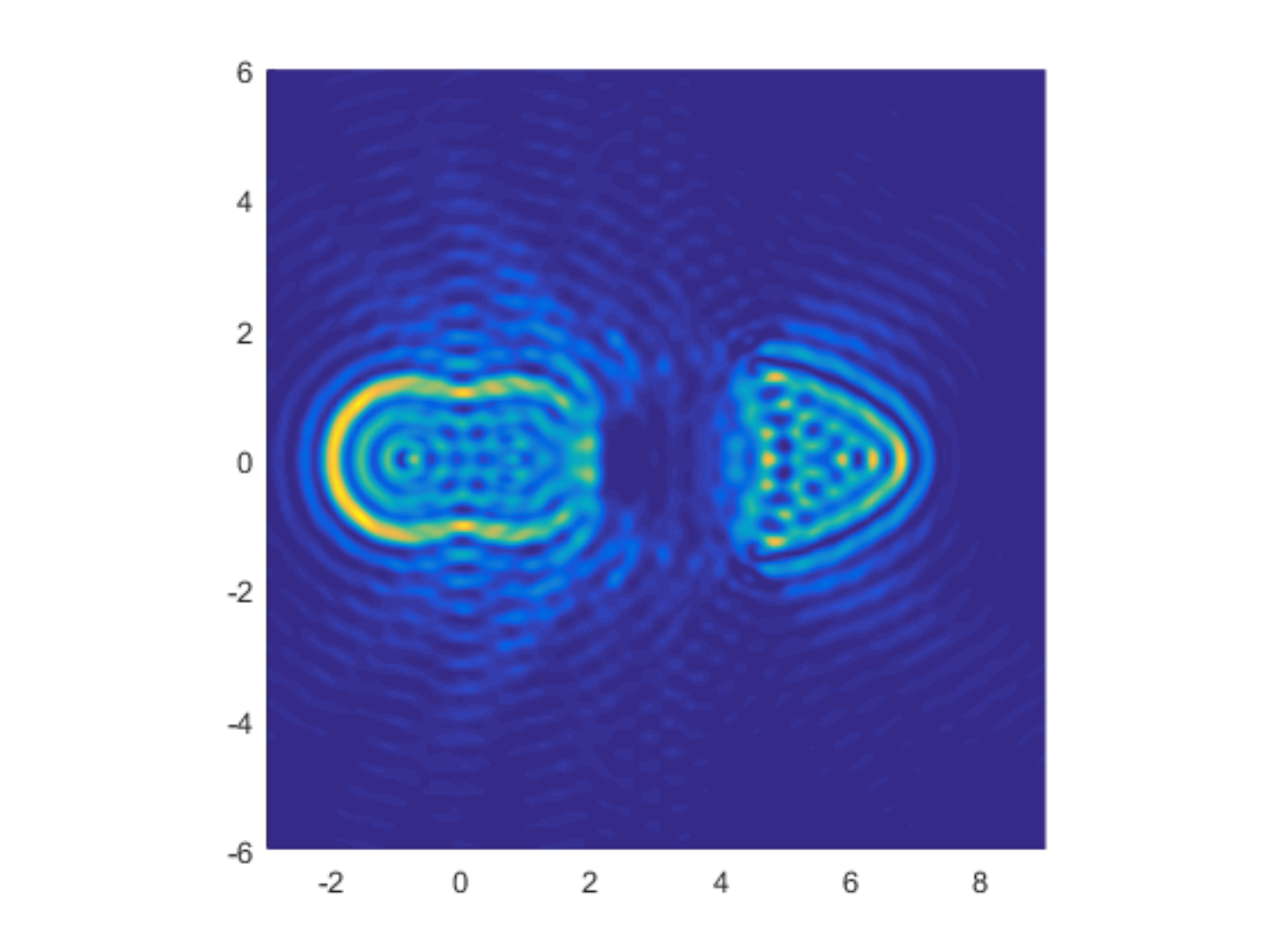}}
\caption{{\bf Example ${\bf I_2}$-Multiple.}\, Reconstruction of mixed type scatterers with different noise.}
\label{2kitepeanut8}
\end{figure}

\textbf{Example ${\bf I_2}$-Multiscalar}.    The scatterer is  the same as the
\textbf{Example ${\bf I_{z_0}}$-Multiscalar}. Figure \ref{2circlesquare8} gives
the results with $10\%, 30\%$ noise.

\begin{figure}[htbp]
  \centering
  \subfigure[\textbf{True domain.}]{
    \includegraphics[width=2in]{pic/pearcirclesoftsoft.pdf}}
  \subfigure[\textbf{$10\%$noise.}]{
    \includegraphics[width=2in]{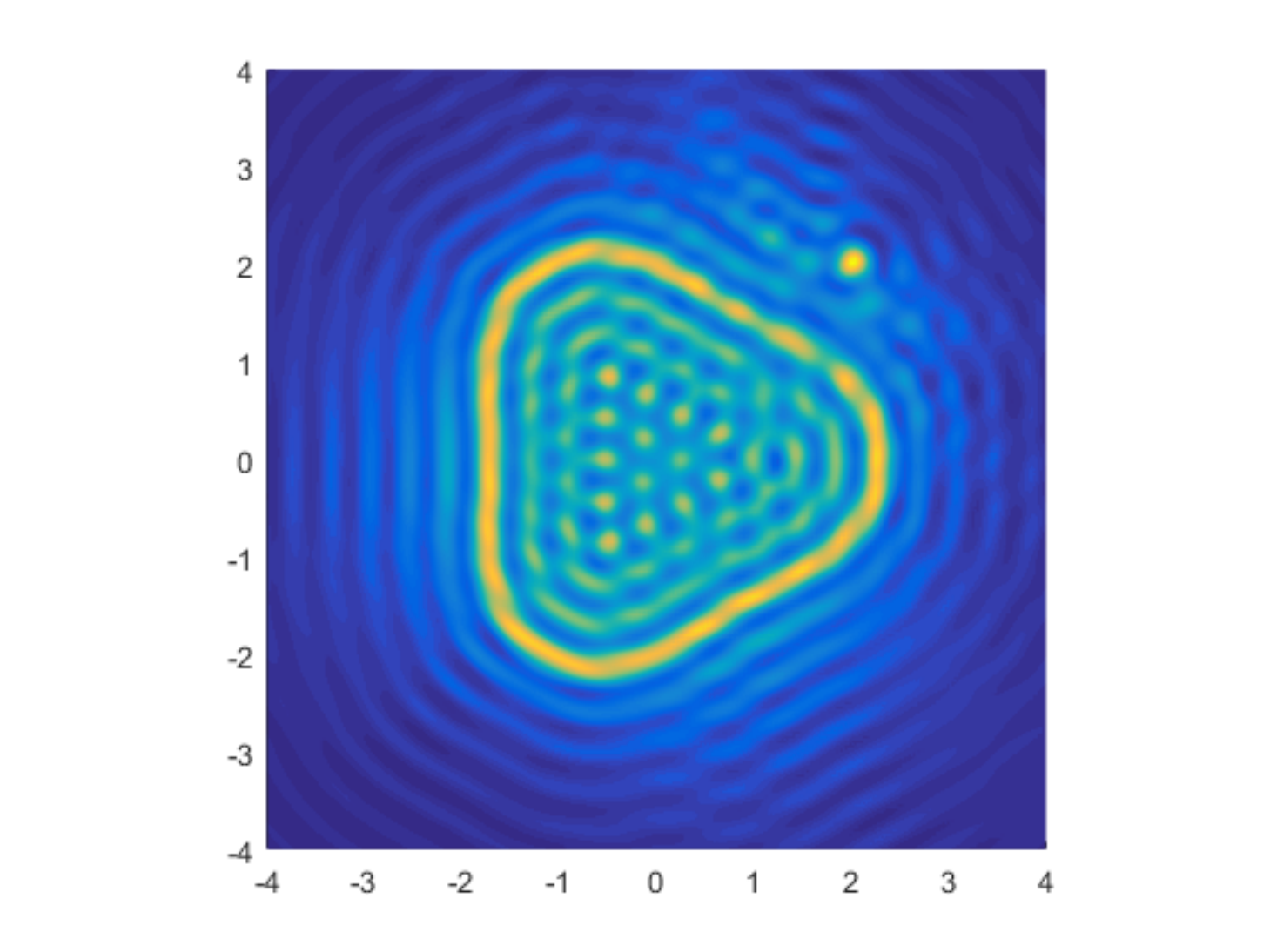}}
  \subfigure[\textbf{$30\%$noise.}]{
    \includegraphics[width=2in]{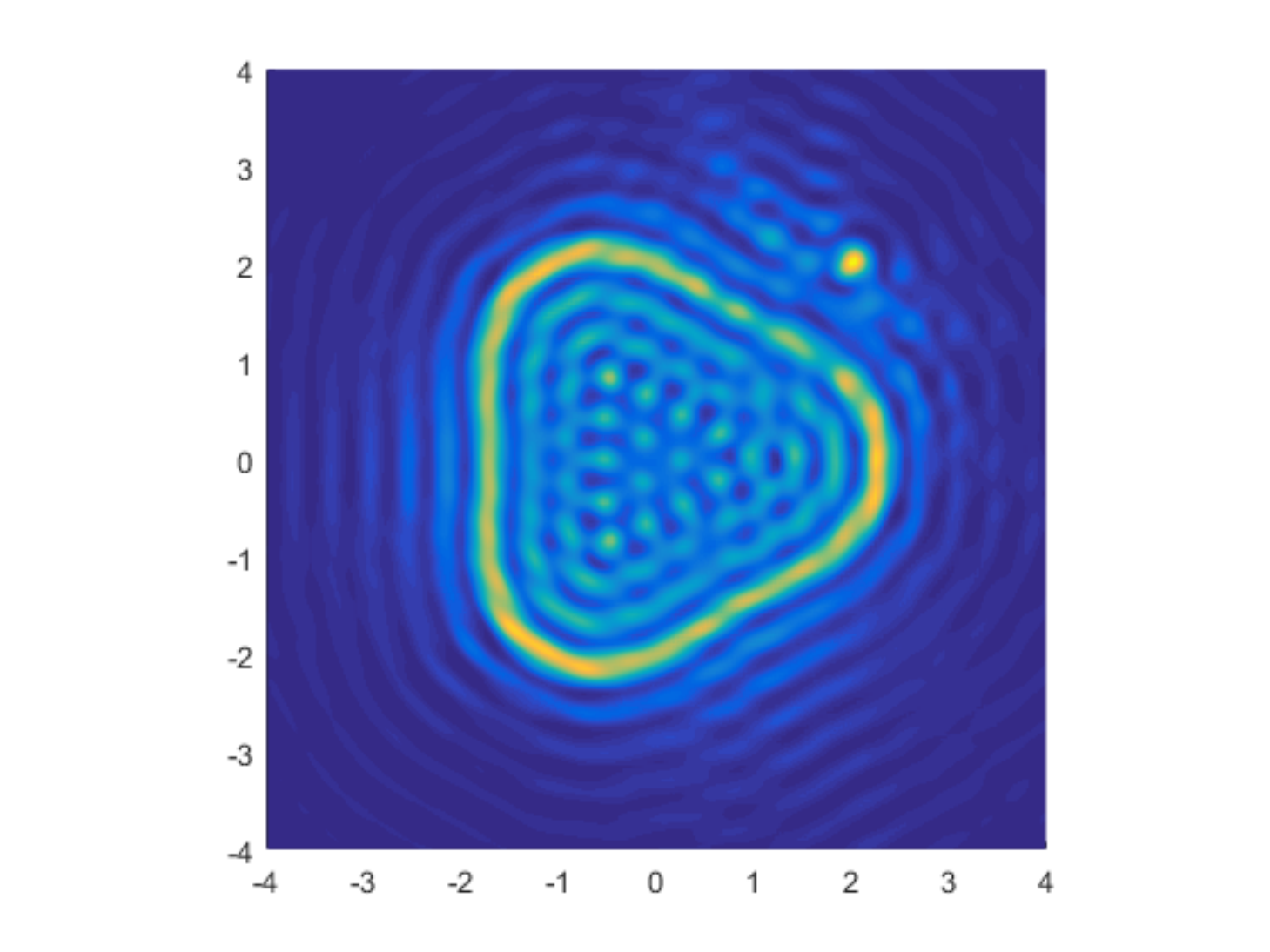}}
\caption{{\bf Example ${\bf I_2}$-Multiscalar.}\, Reconstruction of multiscalar scatterers
with different noise.}
\label{2circlesquare8}
\end{figure}

\textbf{Example ${\bf I_3}$-Small}.
The scatterer is the same as the \textbf{Example ${\bf I^{\Theta}_{z_0}}$-Small}.
Figure \ref{3smallsh8} shows the reconstructions by ${\bf I_3}(z)$ with different incident
directions $(1,0)$ and  $(0,1)$.

\begin{figure}[htbp]
  \centering
  \subfigure[\textbf{True domain.}]{
    \includegraphics[width=2in]{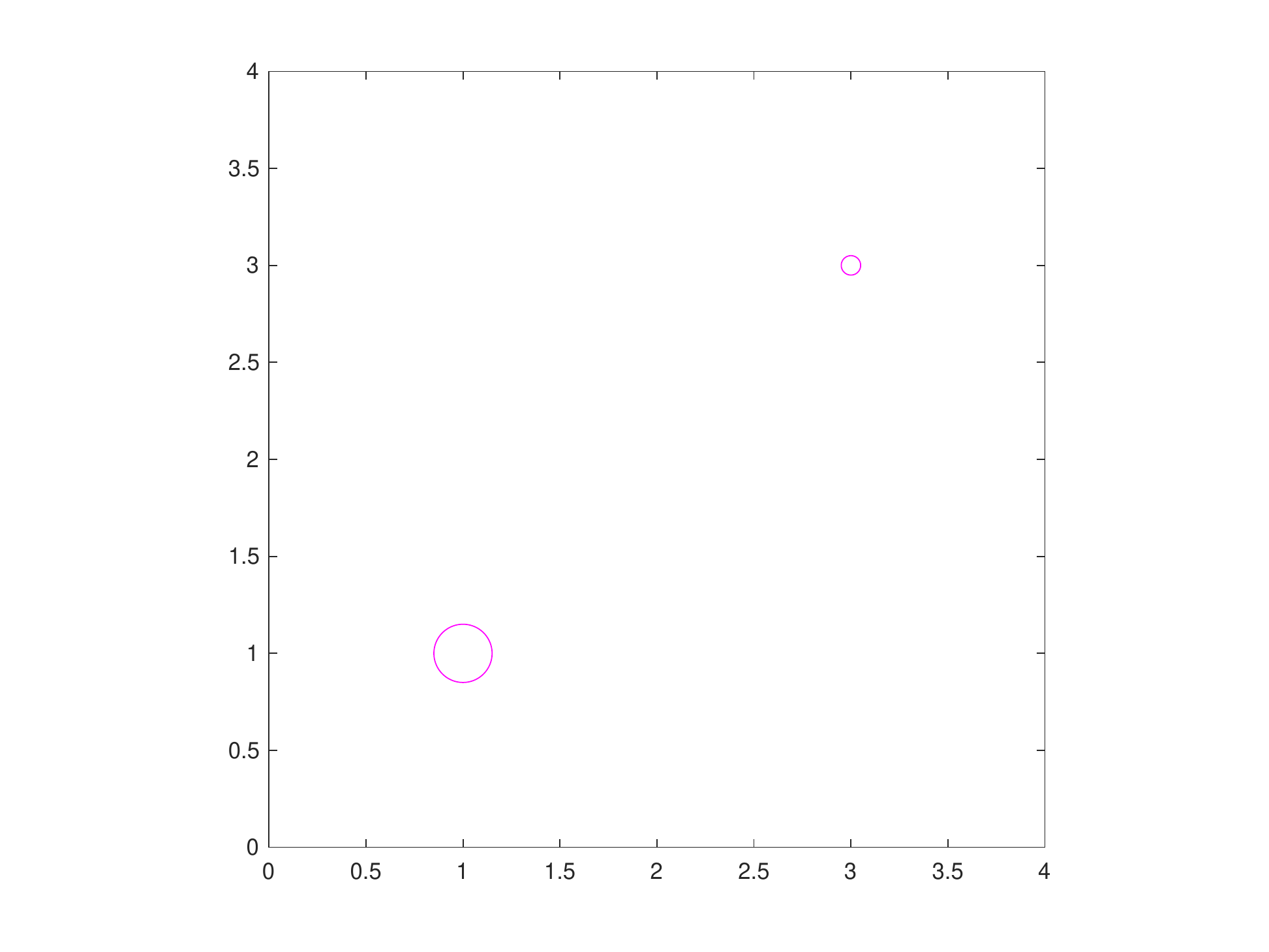}}
  \subfigure[\textbf{$\hth=(1,0)$.}]{
    \includegraphics[width=2in]{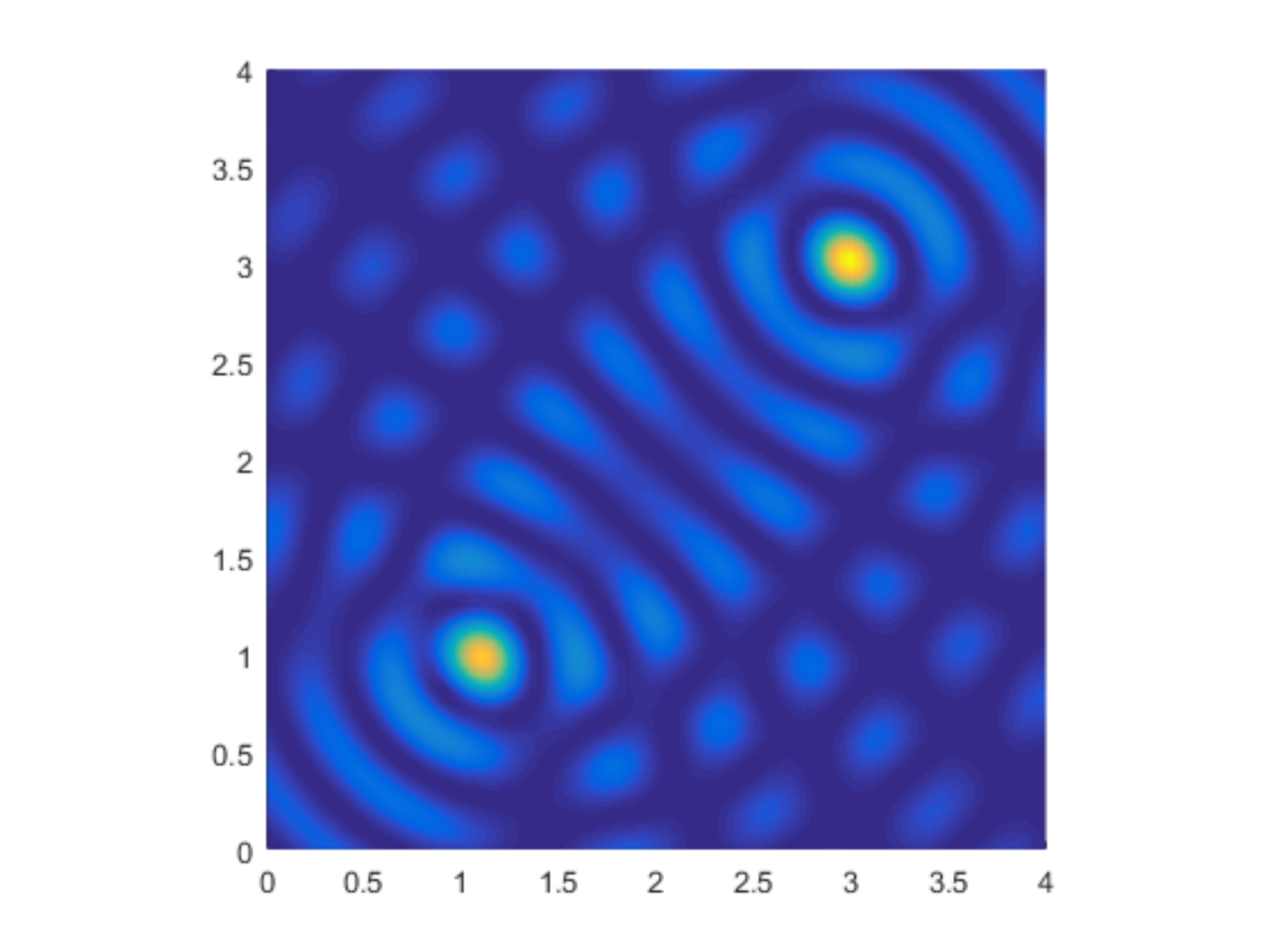}}
  \subfigure[\textbf{$\hth=(0,1)$.}]{
    \includegraphics[width=2in]{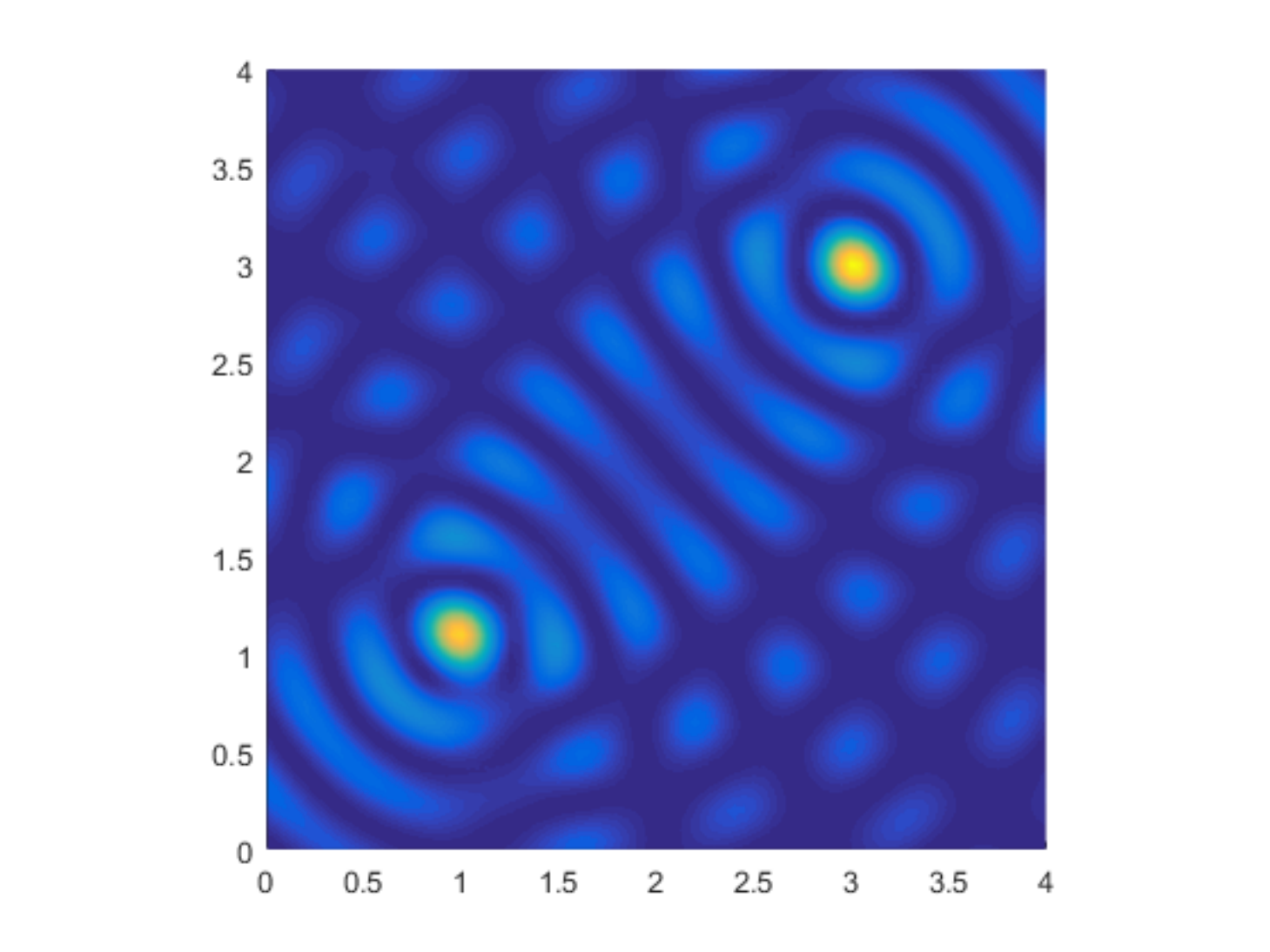}}
\caption{{\bf Example ${\bf I_3}$-Small.}\, Reconstruction of two small mixed type disks
with $10\%$ noise  and different incident waves.}
\label{3smallsh8}
\end{figure}

\section*{Acknowledgement}

The research of X. Ji is partially supported by the NNSF of China under grant 11271018 and 91630313.
and National Centre for Mathematics and Interdisciplinary Sciences, CAS.
The research of X. Liu is supported by the NNSF of China under grant 11571355 and the Youth Innovation
Promotion Association, CAS.
The research of B. Zhang is partially supported by the NNSF of China under grant 91630309.

\bibliographystyle{SIAM}

\end{document}